\documentclass[11pt]{article}
\usepackage{dsfont}
\usepackage[utf8]{inputenc}
\usepackage[T1]{fontenc}
\usepackage[a4paper,margin=25mm]{geometry}
\usepackage[english]{babel}
\usepackage[labelfont=bf]{caption}
\usepackage{subcaption}
\usepackage{array}
\usepackage{booktabs}
\usepackage{amsmath}
\usepackage{amssymb}
\usepackage{amsthm}
\usepackage{mathtools}
\usepackage{mathrsfs}
\usepackage{stackrel}
\usepackage{bbm}
\usepackage[new]{old-arrows}
\usepackage{enumerate}
\usepackage{url}
\usepackage{hyperref}
\usepackage{comment}
\usepackage{csquotes}
\usepackage[numbers]{natbib} 
\newtheorem{Definition}{Definition}[section]
\newtheorem{Proposition}{Proposition}[section]

\newtheorem{Lemma}{Lemma}[section]

\newtheorem{Remark}{Remark}[section]

\numberwithin{equation}{section}

\def\Lc{{\cal L}}
\def\Cc{{\cal C}}

\def\Fc{{\cal F}}

\def\Pc{{\cal P}}

\def\Uc{{\cal U}}

\def \E{\mathbb{E}}

\def \L{\mathbb{L}}
\def \P{\mathbb{P}}

\def \R{\mathbb{R}}

\def\1{{\bf 1}}

\def \trans{^{\scriptscriptstyle{\intercal}}}

\def\*{\times}

\newtheorem{theorem}{Theorem}[section]

\newtheorem{lemma}[theorem]{Lemma}

\theoremstyle{definition}

   \theoremstyle{remark}
\newtheorem{remark}{Remark}
\numberwithin{remark}{section}

\DeclareMathAlphabet{\mathpzc}{OT1}{pzc}{m}{it}

\title{On the convergence of the Euler-Maruyama scheme for McKean-Vlasov SDEs}
\author{
Noufel Frikha\thanks{
Universit\'e Paris 1 Panth\'eon-Sorbonne, Centre d'Economie de la Sorbonne (CES), 106 Boulevard de l’H\^opital, 75642 Paris Cedex 13
({\tt noufel.frikha@univ-paris1.fr}).
},
Xuanye Song\thanks{Universit\'e Paris Cit\'e, CNRS, Laboratoire de Probabilit\'es, Statistique et Mod\'elisation, 75013 Paris ({\tt xsong@lpsm.paris}).},
}
\date{\today}

\def \d{\textnormal{d}}

\def\P{\mathbb{P}}

\hypersetup{
 colorlinks  = true, %
 urlcolor   = black, %
 linkcolor  = black, %
 citecolor  = black %
}
\makeatletter
\renewcommand\@makefnmark{\hbox{\@textsuperscript{\normalfont\color{blue}\@thefnmark}}}
\renewcommand\@makefntext[1]{%
 \parindent 1em\noindent
      \hb@xt@1.8em{%
        \hss\@textsuperscript{\normalfont\@thefnmark}}#1}
\makeatother

\begin{document}

\maketitle

\begin{abstract}
Building on the well-posedness of the backward Kolmogorov partial differential equation in the Wasserstein space, we analyze the strong and weak convergence rates for approximating the unique solution of a class of McKean–Vlasov stochastic differential equations via the Euler–Maruyama time discretization scheme applied to the associated system of interacting particles. We consider two distinct settings. In the first, the coefficients and test function are irregular, but the diffusion coefficient remains non-degenerate. Leveraging the smoothing properties of the underlying heat kernel, we establish the strong and weak convergence rates of the scheme in terms of the number of particles $N$ and the mesh size $h$. In the second setting, where both the coefficients and the test function are smooth, we demonstrate that the weak error rate at the level of the semigroup is optimal, achieving an error of order $N^{-1} + h$.  


 \end{abstract}

\section{Introduction}
Given the probability space $(\Omega,\Fc,\Pc)$, in the current paper, we are interested in the numerical approximation of some non-linear stochastic differential equations (SDEs for short) in the sense of McKean-Vlasov with dynamics
\begin{equation}
\label{SDE:MCKEAN}
X^{\xi}_t = \xi + \int_0^t b(s, X^{\xi}_s, \mathcal{L}(X^{\xi}_s)) \, \d s + \int_0^t \sigma(s, X^{\xi}_s, \mathcal{L}(X^{\xi}_s)) \,  \d W_s, \quad \mathcal{L}(\xi)=\textnormal{Law}(\xi)= \mu \in \mathcal{P}_2(\mathbb{R}^d),
\end{equation}

\noindent where $\xi$ is an $\mathbb{R}^d$-valued random variable which is independent of the $q$-dimensional Brownian motion $W=(W^1,\cdots, W^q)$ and with drift coefficient $b: \mathbb{R}_+ \times \mathbb{R}^d \times \mathcal{P}(\mathbb{R}^d)\rightarrow \mathbb{R}^d$ and diffusion coefficient $\sigma: \mathbb{R}_+ \times \mathbb{R}^d \times \mathcal{P}(\mathbb{R}^d) \rightarrow \mathbb{R}^{d\times q}$, which are both measurable mappings. Here and throughout the paper, $\mathcal{L}(\theta)$ stands for the law of the random variable $\theta$. The well-posedness in the weak or strong sense of the non-linear SDE \eqref{SDE:MCKEAN} has been investigated in numerous works such as G\"artner \cite{gartner}, Sznitman \cite{Sznitman}, Jourdain \cite{jourdain:1997}, Li and Min \cite{Li:min:2},  Mishura and Veretenikov \cite{mishura:veretenikov}, Hammersley et al. \cite{HSSzpruch:18}, Lacker \cite{la:18} and Chaudru de Raynal and Frikha \cite{CHAUDRUDERAYNAL20221} for a short and incomplete sample. 

The standard approximation of the above mean-field SDE is given by the system of $N$ particles $\left\{(X^{i}_t)_{t\in [0,T]}, 1\leq i \leq N\right\}$ interacting through its empirical measure
\begin{equation}
\label{SDE:particle:system}
X^{i}_t = \xi^{i} + \int_0^t b(s, X^{i}_s, \mu^{N}_s) \, \d s + \int_0^t \sigma(s, X^{i}_s, \mu^{N}_s) \,  \d W^{i}_s, \quad \mu^{N}_t := \frac{1}{N} \sum_{i=1}^{N}\delta_{X^{i}_t}, \quad i=1, \cdots, N,
\end{equation}

\noindent where $\left\{(\xi^{i}, (W^{i}_t)_{t\in [0,T]}), 1\leq i \leq N\right\}$ are $N$ i.i.d. copies of the couple $(\xi, W)$. The connection between the two above systems of SDEs comes from fact that the dynamics \eqref{SDE:MCKEAN} describes the limiting behaviour of an individual particle in \eqref{SDE:particle:system} when the size of the population $N$ grows to infinity as stated by the so-called propagation of chaos phenomenon, originally studied by McKean \cite{mckean1967propagation} and then investigated by Sznitman \cite{Sznitman}. 

However, in practical applications, one often wants to approximate the SDE \eqref{SDE:MCKEAN} numerically. Since the mean-field SDE \eqref{SDE:MCKEAN} as well as its associated  system of particles \eqref{SDE:particle:system} are in general not explicitly solvable, one often relies on the so-called Euler-Maruyama time-discretization scheme. For a positive integer $n$, the continuous version of the Euler-Maruyama scheme associated to the system of particles \eqref{SDE:particle:system}, denoted by $\left\{ \mathbf{X}^{N, n}_t = (X^{1, n}_t, \cdots, X^{N, n}_t), 0\leq t \leq T \right\}$, on the time interval $[0,T]$ with fixed mesh size $h=T/n$ is defined as follows for $t \in [0,T]$ and for $i =1, \cdots, N$
\begin{equation}\label{euler:scheme:dynamics}
\begin{cases}
X^{i, n}_t & = X_0^{i} + \int_0^t b(k_n(s), X^{i, n}_{k_n(s)}, \mu^{N, n}_{k_n(s)}) \, \d s + \int_0^t \sigma(k_n(s), X^{i, n}_{\eta_n(s)}, \mu^{N, n}_{k_n(s)}) \, \d W^{i}_s, \\
\mu^{N, n}_t & := \frac{1}{N} \sum_{i=1}^{N}\delta_{X^{i, n}_t},
\end{cases}
\end{equation}
\noindent where for $j=0, \cdots, n$, $k_n(t) = t_j$ if $t_j < t\leq t_{j+1}$ and $t_j = j h$, for any integer $j$.\\

The quantitative error analysis of the mean-field \eqref{SDE:MCKEAN} by the above Euler-Maruyama scheme can be carried out in several ways. The most natural idea is to compare the probability measures $\mathcal{L}(X^{1, n}_t)$ and $\mathcal{L}(X_t^{\xi})$ (or more generally $\mathcal{L}(X^{1, n}_t, \cdots, X^{k, n}_t)$ and $(\mathcal{L}(X_t^{\xi}))^{\otimes k}$ for a fixed $k$) for a given metric on the space of probability measures and for a fixed time $t\in [0,T]$. Another way is to compare the empirical measure $\mu^{N, n}_{t}$ with $\mathcal{L}(X_t^{\xi})$. Many quantitative propagation of chaos results have been established at the level of the particle system \eqref{SDE:particle:system}. For instance, if the diffusion coefficient is constant and the dependence of the drift in the measure variable is of linear form, that is, $b(x, \mu)=\int_{\mathbb{R}^d} B(x,y) \, \mu(\d y)$, for a Lipschitz continuous function $B$, then, it has been proved in \cite{Sznitman} that $\sup_{0\leq t \leq T}\mathcal{W}_2(\mathcal{L}(X_t^{1}),\mathcal{L}(X_t^{\xi}))$ is of order $N^{-\frac12}$. This result has been recently improved in \cite{la:23} in which it is shown that $\mathcal{W}_2(\mathcal{L}(X^{1}_t, \cdots, X^{k}_t),(\mathcal{L}(X_t^{\xi}))^{\otimes k})$ is of order $k/N$. When $b$ and $\sigma$ have a more general structure of dependence with respect to the measure variable and are merely Lipschitz continuous in space and measure uniformly in time then the convergence rate (with respect to the above Wasserstein metric) deteriorates with the dimension $d$ as it is needed to estimate distance between the empirical measure associated to i.i.d. samples and the limiting law for which standard results such as \cite{dereich:13} or \cite{fournier2013rate} can be used. However, remarkably, according to Lemma 5.10 \cite{delarue:lacker:ramanan} or Lemma 3.2 \cite{szpruch:tse}, one may derive a dimension-free rate assuming that the drift and diffusion coefficients satisfy some strong regularity assumption with respect to the measure argument. We also refer to \cite{FRIKHA202176} for a convergence rate (with respect to the Wasserstein metric $\mathcal{W}_1$) for the numerical approximation by the Euler-Maruyama time discretization scheme of some one-dimensional McKean–Vlasov SDE, driven by a spectrally-positive Lévy process and a Brownian motion. \\

Another way to assess the convergence rate is to evaluate the statistical distribution of the empirical measure against functions defined on $\mathcal{P}(\mathbb{R}^d)$, that is, by studying the two quantities
\begin{equation}\label{introduction:strong:weak:semigroup:prop:chaos}
\mathbb{E}[|\Phi(\mu_{t}^{N, n}) - \Phi(\mathcal{L}(X_t^{\xi}))|] \quad \mbox{ and }\quad |\mathbb{E}[\Phi(\mu_{t}^{N, n}) - \Phi(\mathcal{L}(X_t^{\xi}))]|,
\end{equation}

\noindent where $\Phi: \mathcal{P}(\mathbb{R}^d)\rightarrow \mathbb{R}$ satisfies suitable regularity assumptions. The first quantity in \eqref{introduction:strong:weak:semigroup:prop:chaos} should be regarded as a strong error of approximation for $\mathcal{L}(X_t^{\xi})$ by $\mu^{N,n}_{t}$ while the second quantity should be understood as a weak error. This type of convergence has advanced significantly in recent years. In \cite{Oben2019}, the authors prove that the above weak error for $\Phi(\mu)=\int_{\mathbb{R}^d} F(x) \, \mu(\d x)$ is of order $N^{-1}+h$, when both the drift and diffusion coefficients depend on the measure variable through the quantity $\mathbb{E}[\alpha(X^{\xi}_s)]$, for some twice differentiable with bounded derivatives functions $F$ and $\alpha$. In \cite{szpruch:tse}, the authors establish a bound for the weak approximation error of order $N^{-1}+h$ for a general smooth $\Phi$ in the case of constant diffusion coefficient and time-homogeneous smooth drift. One may also refer to \cite{mischler2015new, mischler2013kac}, \cite{kolokoltsov2010nonlinear} for other quantitative weak convergence results for the approximation of $\mu_t$ by the empirical measure of the system of interacting particles \eqref{SDE:particle:system}. As exposed in \cite{cardel19}, the central idea is to leverage the semigroup generated by \eqref{SDE:MCKEAN} that acts on real-valued functions defined on the space of probability measures. This idea has been intensively used in recent works such as \cite{chassagneux:szpruch:tse}, \cite{CHAUDRUDERAYNAL20211} to derive establish weak and strong convergence rate for the empirical measure $\mu_t^{N}$ and \cite{cardaliaguet2019master} to study the convergence problem for mean-field games. We also refer to the recent work \cite{delarue:tse} for a uniform in time weak propagation of chaos result for McKean-Vlasov SDEs with constant diffusion coefficient evolving on the $d$-dimensional torus under some mild-regularity assumption on the drift and test function.   

Our main objective here consists in establishing some explicit strong and weak quantitative convergence rates for the numerical approximation of \eqref{SDE:MCKEAN} by the corresponding Euler-Maruyama scheme \eqref{euler:scheme:dynamics} in two distinct settings.

In the first setting, the coefficients $b$ and $a=\sigma \sigma^{\trans}$ are assumed to be irregular but the diffusion coefficient is non-degenerate. In particular, if $b$ satisfies some mild $\eta$-H\"older regularity assumption in the space and measure variables, for some $\eta \in (0,1]$, (see Section \ref{AssumpSDE} for the precise assumption) and $\sigma$ is Lipschitz-continuous in space and measure variables uniformly in time, we establish a convergence rate for the strong error $\sup_{0\leq t \leq T}\mathcal{W}_2(\mathcal{L}(X_t^{1,n}), \mathcal{L}(X_t^{\xi}))^2$ and $\sup_{0\leq t \leq T}\mathbb{E}[\mathcal{W}_2(\mu_t^{N, n}, \mathcal{L}(X_t^{\xi}))^2]$ which is of the same order (in terms of the number $N$ of particles) as in the standard Cauchy Lipschitz setting. Regarding the convergence rate at the level of the semigroup \eqref{introduction:strong:weak:semigroup:prop:chaos}, we establish an explicit strong convergence rate of order $\mathbb{E}[\mathcal{W}_2(\mu_0^N, \mu)^2]^{1/2} \wedge N^{-1/2} + h^{\eta/2}$ and a weak convergence rate of order $N^{-1} + h^{\eta/2}$ for a suitable class of irregular maps satisfying some mild local H\"older regularity type assumptions. The key idea to establish the aforementioned convergence rates is to leverage the smoothing properties of the mean-field SDE established in \cite{CHAUDRUDERAYNAL20221,CHAUDRUDERAYNAL20211} in the current uniformly elliptic setting.

In the second setting, where both the coefficients and the test function are smooth enough, we prove that the weak rate of convergence for the approximation of $\Phi(\mathcal{L}(X_t^{\xi}))$ by $\Phi(\mu^{N, n}_t)$ is optimal with an error of order $N^{-1}+h$. As already mentioned, this optimal rate of convergence has been established in \cite{Oben2019} for linear functions $\Phi$ when the coefficients have a specific dependence with respect to the measure variable. Also, compared to \cite{szpruch:tse}, we handle the case of a general diffusion coefficient and both coefficients are allowed to be time-dependent.

\subsection*{Outline} The paper is structured as follows. We begin by recalling the central notions of differentiation for maps defined on the space of probability measures, along with the assumptions underpinning our analysis, in Section \ref{section:preliminaries:assumptions:main:results}. In the same section, we present our three main results, which establish various weak and strong convergence rates for approximating the McKean–Vlasov SDE \eqref{SDE:MCKEAN} using its Euler–Maruyama discretization scheme \eqref{euler:scheme:dynamics}. The proofs of these results are then provided in Section \ref{section:proof:mani:results}.

\subsection*{Useful notations}
In the following, we use $C$ and $K$ to denote generic positive constants that may depend on the coefficients $b$ and $\sigma$. The notation $c$ is reserved for constants that depend on $|\sigma|_\infty$ and $\lambda$ (see assumption (\textbf{HE})\, in Section \ref{AssumpSDE}) but not on the time horizon $T$. Moreover, the value of $C$, $K$, and $c$ may vary from line to line.  

We denote by $\mathcal{P}(\mathbb{R}^d)$ the space of probability measures on $\mathbb{R}^d$ and by $\mathcal{P}_q(\mathbb{R}^d) \subset \mathcal{P}(\mathbb{R}^d)$, $q\geq1$, the subset of probability measures with a finite moment of order $q$. We equip $\mathcal{P}(\mathbb{R}^d)$ with the total variation metric. We also let $M_q(\mu)^q := \int_{\mathbb{R}^d} |x|^q \, \mu(\d x)$. For $q\geq1$, we equip $\mathcal{P}_q(\mathbb{R}^d)$ with the $q$-Wasserstein metric defined by
$$
\mathcal{W}_q(\mu, \nu) = \Big( \inf_{\pi \in \Pi(\mu, \nu)}\int_{(\mathbb{R}^d)^2} |x-y|^q \, \pi(\d x, \d y) \Big)^{\frac{1}{q}},
$$

\noindent where $\Pi(\mu, \nu)$ is the set of couplings between $\mu$ and $\nu$, i.e. the set of measures on $(\R^d)^2$ having $\mu,\nu$ as marginal laws.

For a positive variance-covariance matrix $\Sigma$, the function $y\mapsto g(\Sigma , y)$ represents the $d$-dimensional Gaussian kernel with covariance matrix $\Sigma$, given by
$$
g(\Sigma, x) = (2\pi)^{-\frac{d}{2}} (\text{det} \, \Sigma)^{-\frac12} \exp(-\frac12 \langle \Sigma^{-1} x, x \rangle).
$$

We also define the first and second order Hermite polynomials 
$$
H^{i}_1(\Sigma, x) := -(\Sigma^{-1} x)_{i}, \quad H^{i, j}_2(\Sigma, x) := (\Sigma^{-1} x)_i (\Sigma^{-1} x)_j - (\Sigma^{-1} )_{i, j}, \quad 1\leq i, j \leq d.
$$

These polynomials are related to the Gaussian density via the following identities: 
$$
\partial_{x_i} g(\Sigma, x) = H^{i}_1(\Sigma, x) g(\Sigma, x), \quad \partial^2_{x_i, x_j} g(\Sigma, x) = H^{i, j}_2(\Sigma, x) g(\Sigma, x).
$$

 Moreover, when $\Sigma= c I_d$, for some positive constant $c$, we use the simplified notation 
 $$
 g(c, x) := \frac{1}{(2\pi c)^{\frac{d}{2}}} \exp\Big(- \frac{|x|^2}{2c}\Big).
 $$

A key inequality that will be frequently used in this work states that for any for any $p, q>0$ and any $x\in \mathbb{R}$, 
$$
|x|^p e^{-q x^2} \leq \big(\frac{p}{2qe}\big)^{\frac{p}{2}}.
$$

As a direct consequence, we obtain the {\it space-time inequality}: for any $p, \, c>0$ and any $c'>c$ there exists $C:=C(c, c', p)$ such that for any $x \in \mathbb{R}$
\begin{gather}
 |x|^p g(c t, x)\leq C t^{\frac{p}{2}} g(c' t, x). \label{space:time:inequality}
\end{gather}

 This, in turn, yields the following standard Gaussian estimates for the first and second order derivatives of the Gaussian density. Specifically, for any $c>0$ and any $c'>c$, there exists a positive constant $C$ such that
\begin{align}
 |H^{i}_1(c t, x)| g(c t, x) \leq \frac{C}{t^{\frac12}} g(c' t,x), \quad |H^{i, j}_2(c t, x)| g(c t, x) \leq \frac{C}{t}g(c' t, x), \quad 1\leq i,j \leq d. \label{standard}
\end{align}

\section{Preliminaries, assumptions and main results}\label{section:preliminaries:assumptions:main:results}

\subsection{Preliminaries}\label{subsection:preliminaries}
We recall that the space of probability measures $\mathcal{P}(\mathbb{R}^d)$ is endowed with the total variation metric while $\mathcal{P}_q(\mathbb{R}^d)$ is equipped with the $q$-Wasserstein metric.

We here briefly present the regularity notions that we will use when working with mappings defined on $\mathcal{P}_q(\mathbb{R}^d)$. We refer the reader to Lions' seminal lectures \cite{lecture:lions:college}, to Cardaliaguet's lectures notes \cite{cardaliaguet} or to Chapter 5 of Carmona and Delarue's monograph \cite{cardel19} for a more complete and detailed exposition. 

Following the recent works \cite{CHAUDRUDERAYNAL20211, CHAUDRUDERAYNAL20221}, we will employ two notions of differentiation of a continuous map $U$ defined on $\mathcal{P}_q(\mathbb{R}^d)$. The first one, called the {linear functional derivative} and denoted by $\frac{\delta U}{\delta m} $, plays an important role when one investigates the well-posedness of the non-linear martingale problem associated to \eqref{SDE:MCKEAN} as well as the regularity properties of the associated transition density.

\begin{Definition}
\label{continuous:L:derivative} A continuous map $U: \mathcal{P}_q(\mathbb{R}^d) \rightarrow \mathbb{R}$ is said to admit a linear functional derivative if there exists a measurable map $\frac{\delta U}{\delta m}: \mathcal{P}_q(\mathbb{R}^d) \times \mathbb{R}^d \rightarrow \mathbb{R}$ such that for any bounded subset $\mathcal{K} \subset \mathcal{P}_q(\mathbb{R}^d)$
$$
\sup_{(y, m) \in \mathbb{R}^d \times \mathcal{K}} \left\{(1+|y|^q)^{-1}\Big|\frac{\delta U}{\delta m}(m)(y) \Big| \right\}< \infty
$$

\noindent and such that for any $m, m' \in \mathcal{P}_q(\mathbb{R}^d)$,
$$
\lim_{\varepsilon \downarrow 0} \frac{U((1-\varepsilon) m + \varepsilon m') - U(m)}{\varepsilon} = \int_{\mathbb{R}^d} \frac{\delta U}{\delta m}(m)(y) \, \d (m'-m)(y).
$$ 

The map $y\mapsto \frac{\delta U}{\delta m}(m)(y)$ being defined up to an additive constant, we will follow the usual normalization convention $\int_{\mathbb{R}^d} \frac{\delta U}{\delta m}(m)(y) \, \d m(y) = 0$.
\end{Definition}

From the above definition, we observe that for all $m, m' \in \mathcal{P}(\mathbb{R}^d)$
$$
 U(m')-U(m) = \int_0^1 \int_{\mathbb{R}^d} \frac{\delta U}{\delta m}((1-\lambda) m + \lambda m')(y) \,  \d(m'-m)(y)\, \d \lambda.
$$ 
It is readily seen that if $y\mapsto \frac{\delta U}{\delta m}(m)(y)$ is Lipschitz continuous, with a Lipschitz modulus bounded uniformly on a ball $\mathcal{B} \subset \mathcal{P}_q(\mathbb{R}^d)$, then, from the Monge-Kantorovich duality principle for any $m, m' \in \mathcal{B}$
$$
 |U(m) - U(m')| \leq \sup_{m'' \in \mathcal{B}}  \textnormal{Lip}\Big(\frac{\delta U}{\delta m}(m'')(.)\Big) \mathcal{W}_1(m, m').
$$

We will also work with higher order derivatives. This is naturally defined by induction as follows.

\begin{Definition}\label{continuous:L:derivative:higher:order}Let $p\geq2$. The continuous map $U: \mathcal{P}_q(\mathbb{R}^d) \rightarrow \mathbb{R}$ is said to have a linear functional derivative at order $p$ if there exists a measurable map $\frac{\delta^{p} U}{\delta m^{p}}: \mathcal{P}_q(\mathbb{R}^d) \times (\mathbb{R}^d)^{p-1}\times \mathbb{R}^d \rightarrow \mathbb{R}$ such that for any bounded subset $\mathcal{K} \subset \mathcal{P}_q(\mathbb{R}^d) \times (\mathbb{R}^d)^{p-1}$
$$
\sup_{y_p\in \mathbb{R}^d, (m, \mathbf{y}_{p-1}) \in \mathcal{K}}\left\{ (1+|y_p|^q)^{-1} \Big|\frac{\delta^{p} U}{\delta m^{p}}(m)(\mathbf{y}_{p-1}, y_p)\Big| \right\} < \infty
$$ 
\noindent and such that for any $m, m' \in \mathcal{P}_q(\mathbb{R}^d)$ and any $\mathbf{y}_{p-1}\in (\mathbb{R}^d)^{p-1}$
$$
\lim_{\varepsilon \downarrow 0} \varepsilon^{-1} \Big(\frac{\delta^{p-1}U}{\delta m^{p-1}}((1-\varepsilon) m + \varepsilon m')(\mathbf{y}_{p-1}) - \frac{\delta^{p-1}U}{\delta m^{p-1}}(m)(\mathbf{y}_{p-1})\Big) = \int_{\mathbb{R}^d} \frac{\delta^{p} U}{\delta m^{p}}(m)(\mathbf{y}_{p}) \, \d (m'-m)(y_p)
$$
\noindent with the notation $\mathbf{y}_p := (\mathbf{y}_{p-1}, y_p)$. We again follow the convention $\int_{\mathbb{R}^d} \frac{\delta^{p} U}{\delta m^{p}}(m)(\mathbf{y}_p) \, \d m(y_p) = 0$.
\end{Definition}

In the special case where $q=0$, we let $\mathcal{P}_0(\mathbb{R}^q)=\mathcal{P}(\mathbb{R}^q)$ and the linear functional derivative $y_p\mapsto \frac{\delta^p U}{\delta m^p}(m)(\mathbf{y}_{p-1}, y_p)$ is assumed to be bounded locally uniformly in $(m, \mathbf{y}_{p-1})$.\\

The second notion that we will use is the so-called Wasserstein derivative $\partial_\mu U$ (see e.g. \cite{buckdahn2017, cardel19}) of a real-valued map $U$ defined on $\mathcal{P}_2(\mathbb{R}^d)$. Without entering into technical details, according to Proposition 5.48 \cite{cardel19}, it is the gradient field of the linear functional derivative $\partial_\mu U(m)(y) = \partial_y[\delta U/\delta m](m)(y)$. Also, as for the second order Wasserstein derivative, we have $\partial^2_\mu U(m)(y) = \partial_{y'}\partial_y[\delta^2 U/\delta m^2](m)(y, y')$ and we will often use the notation $\partial_\mu^{i} U(m)(y) = \partial_{y_i}[\delta U/\delta m](m) (y)$ and $\partial^{(i,j)}_\mu U(m)(y) = \partial_{(y')_j}\partial_{y_i}[\delta^2 U/\delta m^2](m)(y, y')$, $i, j = 1, \cdots, d$.\\

For our analysis, we will deal with maps defined on $[0,T]\times \mathbb{R}^d\times\mathcal{P}_2(\mathbb{R}^d)$. Hence, we define the space of functions with continuous derivatives not only with respect to the time and space variables but also with respect to the measure variable.

\begin{Definition}\label{DefCp22}
Let $T>0$, $m\in\{0,1\}$. The continuous function $U:[0,T]\times \mathbb{R}^d\times\mathcal{P}_2(\mathbb{R}^d)$ is in $\Cc^{m,2,2}([0,T]\times\mathbb{R}^d\times\mathcal{P}_2(\mathbb{R}^d))$ if the following conditions hold:
\begin{itemize}
    \item [(i)]For any $\mu\in\mathcal{P}_2(\mathbb{R}^d)$, the mapping $[0,T]\times\mathbb{R}^d \ni (t,x)\mapsto U(t,x,\mu)$ belongs to $\Cc^{m,2}([0,T]\times\mathbb{R}^d)$ and the following functions $ [0,T]\times\mathbb{R}^d\times\Pc_2(\mathbb{R}^d) \ni (t,x,\mu)\mapsto \partial^m_t U(t,x,\mu), \, \partial_x U(t,x,\mu)$, $\partial^2_x U(t,x,\mu)$ are continuous.
    \item [(ii)] For any $ (t,x)\in [0,T]\times\mathbb{R}^d$, the mapping $\mathcal{P}_2(\mathbb{R}^d) \ni \mu \mapsto U(t,x,\mu)$ is differentiable in the Wasserstein sense and for any $ \mu\in\mathcal{P}_2(\mathbb{R}^d)$, we can find a version of the mapping $ \mathbb{R}^d \ni \upsilon \mapsto \partial_\mu U(t,x,\mu)(\upsilon)$ such that the mapping $[0,T]\times \mathbb{R}^d\times \mathcal{P}_2(\mathbb{R}^d)\times\mathbb{R}^d \ni (t,x,\mu, \upsilon)\mapsto \partial_\mu U(t,x,\mu)(\upsilon)$ is locally bounded and is continuous at any $(t,x,\mu, \upsilon)$ such that $ \upsilon \in\text{Supp}(\mu)$.
    \item[(iii)] For the above version of $\partial_\mu U$, for any $ (t,x,\mu)\in [0,T]\times\mathbb{R}^d\times\mathcal{P}_2(\mathbb{R}^d)$, the mapping $\mathbb{R}^d\ni \upsilon \mapsto\partial_\mu U(t,x,\mu)(\upsilon)$ is continuously differentiable and its derivative $\partial_\upsilon[\partial_\mu U(t,x,\mu)](\upsilon)\in\mathbb{R}^{d\times d}$ is jointly continuous in $(t,x,\mu, \upsilon)$ at any $(t,x,\mu, \upsilon)$ such that $\upsilon\in\text{Supp}(\mu)$.
\end{itemize}
  The continuous function $U:[0,T]\*\R^d\*\Pc_2(\R^d)$ is in $\Cc_f^{p,2,2}([0,T]\*\R^d\*\Pc_2(\R^d))$ if  $U\in \Cc^{p,2,2}([0,T]\*\R^d\*\Pc_2(\R^d)) $ and the following conditions hold:
  \begin{itemize}
      \item [(iv)] For all $ \upsilon\in\mathbb{R}^d$, the version $\mu\in\mathcal{P}_2(\mathbb{R}^d)\mapsto\partial_\mu U(t,x,\mu)(\upsilon)$ discussed in (ii) is differentiable component by component with a derivative given by 
    $\mathcal{P}_2(\mathbb{R}^d)\times(\mathbb{R}^d)^2 \ni (\mu, \upsilon , \upsilon')\mapsto \partial^2_\mu U(t,x,\mu)(\upsilon, \upsilon')\in\mathbb{R}^{d\times d}$ such that for all $ \mu\in\mathcal{P}_2(\mathbb{R}^d)$ and $X\in\L_2$ with $\Lc(X)=\mu$, the $\mathbb{R}^{d\times d}$-valued random variable $\partial^2_\mu U(t,x,\mu)(\upsilon)(X)$ gives the Fréchet derivative of the map $X'\in\L_2\mapsto\partial^2_\mu U(t,x,\Lc(X'))(\upsilon)$ for all $\upsilon\in\mathbb{R}^d$. The map $(t,x,\mu, \upsilon, \upsilon')\in [0,T]\times \mathbb{R}^d\times\mathcal{P}_2(\mathbb{R}^d)\times(\mathbb{R}^d)^2\mapsto \partial^2_\mu U(t,x,\mu)(\upsilon, \upsilon')$ is also continuous for the product topology.
  \end{itemize}
\end{Definition}
\subsection{Main assumptions}\label{AssumpSDE}

As previously mentioned, we will distinguish two different settings in our analysis. The first one corresponds to the case where the coefficients $b$ and the diffusion matrix $\mathbb{R}^{d \times d}\owns a(t,x,m)=\sigma\sigma\trans(t,x,m),(t,x,m)\in [0,T]\times\mathbb{R}^d\times\mathcal{P}(\mathbb{R}^d)$ are irregular but $a$ is non-degenerate. In this setting, we introduce the following two assumptions. \\

\noindent(\textbf{HR})
\begin{itemize}
    \item [(i)]The maps $\mathbb{R}_+\times\mathbb{R}^d\times\mathcal{P}(\mathbb{R}^d)\owns (t,x,m)\mapsto b(t,x,m)\in\mathbb{R}^d,a(t,x,m)\in\mathbb{R}^{d\times d}$ are bounded. Moreover, for any $m \in  \mathcal{P}(\mathbb{R}^d)$ and any $(i,j)\in \left\{1, \cdots, d\right\}^2$, $(t, x) \mapsto b_i(t, x, m), \, a_{i, j}(t, x, m)$ are $(\eta/2,\eta)$-H\"older continuous, for some $\eta\in (0,1]$, namely
    $$
    \sup_{m, \, (t,x)\neq (s,y)}\frac{|b_i(t, x, m)-b_i(s, y, m)| + |a_{i, j}(t, x, m)-a_{i, j}(s, y, m)|}{(t-s)^{\eta/2} + |x-y|^\eta}< \infty.
    $$
    \item[(ii)]For any $(i,j)\in\{1,\dots,d\}^2$ and any $(t,x)\in\mathbb{R}_+\times \mathbb{R}^d$, the map $m\mapsto a_{i, j}(t,x,m)$ admits two bounded and  continuous linear functional derivatives, such that $(x,\upsilon)\mapsto\frac{\delta a_{i, j}}{\delta m}(t,x,m)(\upsilon)$ and $(x, \upsilon')\mapsto \frac{\delta^2 a_{i, j}}{\delta m^2}(t,x,m)(\upsilon, \upsilon')$ are $\eta$-H\"older continuous function, for some $\eta\in(0,1]$ uniformly with respect to the other variables. 
    \item[(iii)]For any $i\in\{1,\dots,d\}$ and any $(t,x)\in\mathbb{R}_+\times\mathbb{R}^d$, the map $m\mapsto b_i(t,x,m)$ admits two bounded and continuous linear functional derivatives, such that $\upsilon \mapsto\frac{\delta b_{i}}{\delta m}(t,x,m)(\upsilon)$ and $\upsilon'\mapsto \frac{\delta^2 b_i}{\delta m^2}(t,x,m)(\upsilon, \upsilon')$ are $\eta$-H\"older continuous function, for some $\eta\in(0,1]$, uniformly with respect to the other variables. 
\end{itemize}
(\textbf{HE}) The diffusion matrix $a$ is uniformly non-degenerate, that is, there exists $\lambda\geq 1$ such that for any $m\in\mathcal{P}(\mathbb{R}^d)$ and any $(t,x, z)\in\mathbb{R}_+\times (\mathbb{R}^d)^2$ 
$$
\lambda^{-1}\lvert z \rvert^2\leq \langle a(t,x,m)z,z\rangle\leq\lambda\lvert z \rvert^2.
$$

According to Theorem 3.4 in \cite{CHAUDRUDERAYNAL20221}, under the assumptions (\textbf{HR}) and (\textbf{HE}), the SDE \eqref{SDE:MCKEAN} is well-posed in the weak sense for any initial distribution $\mu\in\mathcal{P}(\mathbb{R}^d)$. Moreover, weak well-posedness also holds when the SDE \eqref{SDE:MCKEAN} starts from the random vector $\xi$ with law $\mu$ at any given time $s\geq0$.

Introducing the decoupled stochastic flow associated to the SDE \eqref{SDE:MCKEAN} given by  the unique solution of the following SDE with dynamics
\begin{equation}\label{SDEEMdecop}
    X^{s,x,\mu}_t=x+\int_s^t b(r,X^{s,x,\mu}_r,\mathcal{L}(X^{s,\xi}_r)) \, \d r+\int_s^t \sigma (r,X^{s,x,\mu}_r,\mathcal{L}(X^{s,\xi}_r)) \, \d W_r,
\end{equation}

\noindent recalling that $\mathcal{L}(X^{s,\xi}_t)$ stands for the law of the solution of SDE \eqref{SDE:MCKEAN} (starting at time $s$ from $\xi$ with distribution $\mu$) at time $t$, it follows that, under (\textbf{HR}) and (\textbf{HE}), both random variables $X^{s,x,\mu}_t$ and $X^{s,\xi}_t$ admit a density, for any $t>0$, denoted by $z\mapsto p(\mu,s,t,x,z)$ and $p(\mu,s,t,z)$ respectively which satisfy
$$
p(\mu,s,t,z)=\int_{\mathbb{R}^d}p(\mu,s,t,x,z) \, \mu(\d x).
$$

Moreover, Theorem 3.6 \cite{CHAUDRUDERAYNAL20221} guarantees that for any fixed $t>0$ and $z\in \mathbb{R}^d$, $(s, \mu) \mapsto p(\mu, s, t, z) \in \mathcal{C}^{1, 2}([0,t)\times \mathcal{P}_2(\mathbb{R}^d))$ and $(s, x, \mu) \mapsto p(\mu, s, t, x, z) \in \mathcal{C}^{1, 2, 2}([0,t)\times \mathbb{R}^d \times \mathcal{P}_2(\mathbb{R}^d))$ and that their derivatives satisfies some Gaussian estimates.

Let us finally emphasize that assumptions (\textbf{HR})(ii) and (\textbf{HR})(iii) ensure that for any $(t, x) \in \mathbb{R}_+\times \mathbb{R}^d$, both mappings $m\mapsto b_i(t, x, m), \, a_{i, j}(t, x, m)$ are uniformly $\eta$-H\"older continuous with respect to the Wasserstein metric $\mathcal{W}_1$. Combined with (\textbf{HR})(i), this further implies that $b_i$ and $a_{i, j}$ satisfy the following H\"older regularity estimate: there exists a constant $C<\infty$ such that for any $t, s\in \mathbb{R}_+$, any $x, y \in \mathbb{R}^d$ and any $m, m' \in \mathcal{P}(\mathbb{R}^d)$, we have
\begin{equation}\label{holder:reg:coefficients}
\begin{aligned}
|b_i(t, x, m) - b_i(s, y, m')| & + |a_{i, j }(t, x, m) - a_{i, j}(s, y, m')| \\
& \leq C \Big( (t-s)^{\eta/2} + |x-y|^\eta + \mathcal{W}_1(m, m')^\eta \Big).
\end{aligned}
\end{equation}

The second setting corresponds to the case of smooth coefficients $b$ and $\sigma$ in which we establish the optimal weak convergence rate at the level of the semigroup. We introduce the following assumption.\\
\noindent(\textbf{HFR})
\begin{itemize}
    \item [(i)] For any $1\leq i, \, j\leq d$, the coefficients $b_i, \, a_{i, j} \in \mathcal{C}^{1, 2, 2}_f([0,T] \times \mathbb{R}^d \times \mathcal{P}_2(\mathbb{R}^d))$ with the corresponding derivatives being bounded. 
    \item[(ii)] There exist a positive integer $\tilde{q}$ and $\tilde{\sigma}: [0,T]\times \mathbb{R}^d \times \mathcal{P}_2(\mathbb{R}^d) \rightarrow \mathbb{R}^d \times \mathbb{R}^{\tilde{q}} $ such that $a=\tilde{\sigma}\tilde{\sigma}^{t}$. For any $f \in \left\{ b_i, \, \tilde{\sigma}_{i, j}, \, 1\leq i \leq d, \, 1\leq j\leq \tilde{q} \right\}$ and any $1\leq k, \ell \leq d$, the derivatives $\partial_{x_k} f$, $\partial^2_{x_k, x_\ell} f$, $\partial^{k}_\mu f$, $\partial^{(k, \ell)}_\mu f$, $\partial_{x_\ell} \partial^{k}_\mu f$ and $\partial^{\ell}_\mu \partial_{x_k} f$ exist, are continuous, bounded with the second order derivatives $\partial^2_{x_k, x_\ell} f(s,.)$, $\partial^{(k, \ell)}_\mu f(s,.)(.)$, $\partial_{x_\ell} \partial^{k}_\mu f(s,.)(.) $ and $\partial^{\ell}_\mu \partial_{x_k} f(s,.)(.)$ being Lipschitz continuous uniformly in $s\in [0,T]$.
\end{itemize}

\subsection{Main results}

Our first main result quantifies the strong convergence rate of the Euler-Maruyama scheme \eqref{euler:scheme:dynamics} at level of the trajectories. As in the standard Cauchy-Lipschitz framework, we leverage a coupling argument based on the auxiliary process $\left\{\bar{X}^i=(\bar{X}^i_t)_{0\leq t\leq T}, 1\leq i\leq N\right\}$ following the same dynamics as the McKean-Vlasov SDE \eqref{SDE:MCKEAN} but with the same input $(\xi^{(i)},W^{(i)})_{1\leq i\leq N}$ as the Euler-Maruyama scheme \eqref{euler:scheme:dynamics} instead of $(\xi,W)$:
\begin{equation}\label{SDEbarXi}
    \bar{X}^i_t  = \xi^{(i)}+\int_0^t b(s,\bar{X}^i_s, \mathcal{L}(\bar{X}^{i}_s)) \, \d s+\int_0^t\sigma(s,\bar{X}^i_s, \mathcal{L}(\bar{X}^{i}_s)) \, \d W^{(i)}_s, \quad 1\leq i \leq N.
\end{equation}

By weak uniqueness, the processes $\left\{(\bar{X}^i_t)_{0\leq t \leq T}, i=1,\cdots,N\right\}$ are i.i.d. copies of $(X^{\xi}_t)_{0\leq t \leq T}$. In particular, for all $t\in [0,T]$ and all $i \in \left\{1, \cdots, N\right\}$, $\mathcal{L}(\bar{X}^{i}_t) = \mathcal{L}({X}^{\xi}_t)$.

The proof of the following result is postponed to Section \ref{proof:Thm:Conv:EulerMaru}.
\begin{theorem}[Strong rate of convergence at the level of the trajectories]\label{Thm:Conv:EulerMaru}
    Assume that (\textbf{HE}) and (\textbf{HR}) are satisfied and that $M_q(\mu) = (\int |x|^q \, \mu(\d x))^{1/q} <+\infty$ for some $q>4$. Assume that $\mathbb{R}_+\times \mathbb{R}^d \times \mathcal{P}_2(\mathbb{R}^d) \ni (t, x, m)\mapsto \sigma(t, x, m)$ is $\eta/2$-H\"older continuous in time (uniformly in $x$, $m$) and Lipschitz continuous in $x$ and $m$ (uniformly in $t$). Then, there exists a constant $C=C(T, (\textbf{HR}), (\textbf{HE}), M_q(\mu))<\infty$, $T\mapsto C(T, (\textbf{HR}), (\textbf{HE}), M_q(\mu))$ being non-decreasing, such that 
    \begin{equation}\label{strong:error:sup:outisde}
        \sup_{0\leq t\leq T}\E[\mathcal{W}_2(\mathcal{L}(X^{\xi}_t),\mu^{N,n}_t)^2]+\max_{i=1,\dots,N}\sup_{0\leq t\leq T}\E[\lvert X^{i,n}_t-\bar{X}^i_t\rvert^2]\leq C(\varepsilon_N+h^\eta),
    \end{equation}
    and 
     \begin{equation}\label{strong:error:sup:inside}
       \E[ \sup_{0\leq t\leq T}\mathcal{W}_2(\mathcal{L}(X^{\xi}_t),\mu^{N,n}_t)^2]+\max_{i=1,\dots,N}\E[\sup_{0\leq t\leq T}\lvert X^{i,n}_t-\bar{X}^i_t\rvert^2]\leq C(\sqrt{\varepsilon_N}+h^\eta),
    \end{equation}
    where $\varepsilon_N$ is defined by 
 \begin{equation}\label{def:epsilonN}
     \varepsilon_N=\begin{cases}
         N^{-\frac{1}{2}}, \quad d<4\\
         N^{-\frac{1}{2}}\log(N+1), \quad d=4\\
         N^{-\frac{2}{d}}, \quad d>4.
     \end{cases}
 \end{equation}
\end{theorem}

Inspired by Theorem 3.6 \cite{CHAUDRUDERAYNAL20211}, we also establish a weak and strong convergence rate for the difference between the marginal law associated to \eqref{SDE:MCKEAN} and its approximation by the empirical measure of the system of particles \eqref{euler:scheme:dynamics} when both act on a general class of irregular maps $\Phi$ defined on $\mathcal{P}_2(\mathbb{R}^d)$. The proof of the following result is deferred to Section \ref{proof:Thm:Conv:EulerMaru:semigroup}.

\begin{theorem}[Weak and strong rates of convergence at the level of the semigroup]\label{Thm:Conv:EulerMaru:semigroup}
Assume that (\textbf{HR}) and (\textbf{HE}) hold. Let $\mathscr C^{2,\alpha}(\mathcal{P}_2(\mathbb{R}^d),\mathbb{R})$ be the set of continuous functions $\Phi : \mathcal{P}_2(\mathbb{R}^d) \rightarrow \mathbb{R}$ such that 
$$
\mathcal{P}_2(\mathbb{R}^d) \times \mathbb{R}^d \ni (m, \upsilon) \mapsto \frac{\delta \Phi}{\delta m}(m)(\upsilon) \quad \mbox{ and } \quad \mathcal{P}_2(\mathbb{R}^d) \times (\mathbb{R}^d)^2 \ni (m, \upsilon, \upsilon')\mapsto \frac{\delta^2 \Phi}{\delta m^2}(m)(\upsilon, \upsilon')
$$
\noindent exist, are continuous and satisfies the following regularity and growth properties: there exists $C<\infty$ such that for any $(m, \upsilon, \upsilon') \in \mathcal{P}_2(\mathbb{R}^d) \times (\mathbb{R}^d)^2$ and any bounded set $D \subset \mathbb{R}^d$,
\begin{align}
\sup_{ \upsilon, \upsilon' \in D, \, \upsilon \neq \upsilon' } \frac{|[\delta \Phi/\delta m](m)(\upsilon) - [\delta \Phi/\delta m](m)(\upsilon')|}{|\upsilon-\upsilon'|^{\alpha}} & \leq C (1 + M_2(m)), \label{local:holder:reg:first:order:linear:functional:deriv:phi} 
\end{align}
\begin{align}
\sup_{ \upsilon', \upsilon'' \in D, \, \upsilon' \neq \upsilon'' } \frac{|[\delta^2 \Phi/\delta m^2](m)(\upsilon, \upsilon') - [\delta^2 \Phi/\delta m^2](m)(\upsilon, \upsilon'')|}{|\upsilon'-\upsilon''|^{\alpha}} & \leq C (1 + |\upsilon| + M_2(m)), \label{local:holder:reg:second:order:linear:functional:deriv:phi} 
\end{align}
\noindent and
\begin{align}
\Big| \frac{\delta \Phi}{\delta m}(m)(\upsilon) \Big| + \Big| \frac{\delta^2 \Phi}{ \delta m^2}(m)(\upsilon, \upsilon') \Big|  & \leq C (1 + |\upsilon| + |\upsilon'| + M_2(m)). \label{growth:condition:deriv:phi}
\end{align}

 Then, there exist two constant $ C:=C(T, M_4(\mu), (\textbf{HR}), (\textbf{HE}))$ (which is finite if $M_4(\mu)=(\int_{\mathbb{R}^d} |x|^4 \, \mu(\d x))^{1/4}< \infty$) and $K(T, M_2(\mu), )<\infty$, $T\mapsto C(T, M_4(\mu), (\textbf{HR}), (\textbf{HE}))$, $K(T, M_2(\mu), (\textbf{HR}), (\textbf{HE}))$ being non-decreasing, such that for all $\Phi \in \mathscr C^{2, \alpha}(\mathcal{P}_2(\mathbb{R}^d),\mathbb{R})$ and all $t\in[0,T]$:
\begin{align}
|\mathbb{E}[\Phi(\mu_t^{N, n})] -\Phi(\mathcal{L}(X^{\xi}_t))|& \leq C \left(t^{\frac{\alpha}{2}-1} \frac{1}{N}  +  t^{\frac \alpha 2} h^{\frac{\eta}{2}}\right), \label{esti:sem:chaos}\\
\mathbb{E}\left[|\Phi(\mu_t^{N, n}) -\Phi(\mathcal{L}(X^{\xi}_t))|\right] &\leq  K \left(t^{\frac{\alpha-1}{2}} \mathbb{E}\Big[\mathcal{W}_2(\mu_0^{N}, \mu)^2\Big]^{\frac12} +  t^{\frac{\alpha}{2}} \big(h^{\frac{\eta}{2}} + \frac{1}{\sqrt{N}} \big) \right).  \label{strong:conv:chaos}
\end{align}
\end{theorem}

\begin{remark} For $\alpha=1$, 
$$
\big\{\Phi : \mathcal{P}(\mathbb{R}^d)\to \mathbb{R},\, \Phi(\mu) = \int_{\mathbb{R}^d} \varphi(x) \, \d\mu(x),\, \varphi \, {\rm is\ 1-Lipschitz}\big\} \subset \mathscr C^{2, 1}(\mathcal{P}_2(\mathbb{R}^d),\mathbb{R}^d)
$$ 

\noindent so that the inequality \eqref{esti:sem:chaos} implies the convergence in first order Wasserstein distance owing to the Kantorovitch-Rubinstein duality theorem:
$$
\mathcal{W}_1(\mathbb{E}[\mu^{N, n}_t], \mu_t) \leq C \left( t^{-\frac12} N^{-1}+ t^{\frac{1}{2}} h^{\frac{\eta}{2}} \right).
$$
\end{remark}

Our final result pertains to the weak error at the semigroup level when the test function $\Phi$ and the coefficients are sufficiently smooth. In this setting, we establish that the weak convergence rate is optimal, achieving an error of order $N^{-1}+h$. The proof of our last result is deferred to Section \ref{proof:thm:error:upper:bound:smooth:case}.

\begin{theorem}\label{thm:error:upper:bound:smooth:case}
Assume that (\textbf{HFR}) holds. Let $\Phi:\mathcal{P}_2(\mathbb{R}^d)\rightarrow \mathbb{R}$ be such that the derivatives
$$
\partial_\mu^{i}\Phi(\mu)(\upsilon_1), \quad \partial_{\upsilon_j} \partial^{i}_\mu \Phi(\mu)(\upsilon_1), \quad \partial^{(i, j)}_{\mu} \Phi(\mu)( \upsilon_1, \upsilon_2), \quad 1\leq i, j \leq d,
$$

\noindent exist, are continuous and bounded on $\mathcal{P}_2(\mathbb{R}^d)\times \mathbb{R}^d$ or $\mathcal{P}_2(\mathbb{R}^d)\times (\mathbb{R}^d)^2$. Assume additionally that the derivatives $ \mathcal{P}_2(\mathbb{R}^d)\times \mathbb{R}^d \ni (\mu, \upsilon) \mapsto \partial^{i}_\mu \Phi(\mu)(\upsilon), \, \partial_{\upsilon_j} \partial^{i}_\mu \Phi(\mu)(\upsilon)$ are Lipschitz continuous. 

Then, for any $T>0$, there exists a constant $C:=C(T, M_2(\mu), (\textbf{HFR}))$, where  $T\mapsto C(T, M_2(\mu), (\textbf{HFR}))$ is non-decreasing and depends also on the aforementioned derivatives of $\Phi$, such that
\begin{equation}\label{error:upper:bound:smooth:case}
\sup_{0\leq t \leq T} |\mathbb{E}[\Phi(\mu^{N,n}_t)-\Phi(\mathcal{L}(X^{\xi}_t))]| \leq C (N^{-1} +h).
\end{equation}
\end{theorem}


\section{Proof of the main results}\label{section:proof:mani:results}
\subsection{Auxiliary results for the proof of Theorem \ref{Thm:Conv:EulerMaru}}\label{AuxResMaruyama}
In the spirit of \cite{CHAUDRUDERAYNAL20221}, the central idea of the proof of Theorem \ref{Thm:Conv:EulerMaru} is to leverage the unique solution $\mathscr{V}:(t,x,\mu)\in[0,T]\times \mathbb{R}^d\times\mathcal{P}_2(\mathbb{R}^d)\mapsto \mathbb{R}^d$ of the following backward Kolmogorov PDE on the Wasserstein space:
 \begin{equation}\label{PDEU}
(\partial_t+\mathcal{L}_t + \mathscr{L}_t) \mathscr{V}(t,x,\mu)=b(t,x,\mu), \quad \mathscr{V}(T,\cdot,\cdot)=0_d
 \end{equation}
 
 \noindent where the operators $\mathcal{L}_t$ and $\mathscr{L}_t$ are defined  by
 \begin{equation}\label{Lt}
     \begin{aligned}
         \mathcal{L}_t g(t,x,\mu)&=\sum_{i=1}^d b_i(t,x,\mu)\partial_{x_i}g(t,x,\mu)+\frac{1}{2}\sum_{i,j=1}^da_{i,j}(t,x,\mu)\partial^2_{x_i x_j}g(t,x,\mu)\\
        \mathscr{L}_t g(t,x,\mu) & = \int_{\mathbb{R}^d}\Big\{ \sum_{i=1}^d b_i(t,\upsilon,\mu)\partial^{i}_\mu g(t,x,\mu)(\upsilon) \\
        & \quad +\frac{1}{2}\sum_{i,j=1}^da_{i,j}(t,z,\mu)\partial_{\upsilon_j}\partial^i_\mu g(t,x,\mu)(z)\Big\} \, \mu(\d \upsilon).
     \end{aligned}
 \end{equation}

It follows from Theorem 3.8 in \cite{CHAUDRUDERAYNAL20221} that there exists a unique solution $\mathscr{V}\in\mathcal{C}^{1,2,2}([0,T)\times\mathbb{R}^d\times \mathcal{P}_2(\mathbb{R}^d)) \cap \mathcal{C}([0,T]\times \mathbb{R}^d \times \mathcal{P}_2(\mathbb{R}^d))$ to the PDE \eqref{PDEU} which satisfies for any $ (t,x,\mu)\in [0,T)\times \mathbb{R}^d\times\mathcal{P}_2(\mathbb{R}^d)$, 
$$
\mathscr{V}(t,x,\mu)=\mathbb{E}\Big[\int_t^T b(s,X^{t,x,\mu}_s,\mathcal{L}(X^{t,\xi}_s)) \, \d s \Big]=\int_t^T \int_{\mathbb{R}^d} b(s, z, \mathcal{L}(X^{t, \xi}_s)) \, p(\mu, t, s, x, z) \, \d s\d z, 
$$

\noindent recalling that $z\mapsto p(\mu,s,t, x, z)$ stands for the density function of $X_s^{t, x, \mu}$.

Unfortunately, as for the existence of the second order derivative $\partial^2_\mu \mathscr{V}$ which is strongly required for our purpose, there is no guarantee that it exists under the sole assumptions \textbf{(HE)} and \textbf{(HR)}. As proposed in \cite{CHAUDRUDERAYNAL20211}, the key idea is to proceed using the following approximation argument that we now briefly present. For an integer $m$, we consider the following SDE with dynamics:
\begin{equation}\label{Xm1}
    X^{s,\xi, (m+1)}_t = \xi+\int_s^t b(r,X^{s,\xi, (m+1)}_r,\mathcal{L}(X^{s, \xi, (m)}_r)) \, \d r+\int_s^t \sigma(r,X^{s,\xi, (m+1)}_r, \mathcal{L}(X^{s, \xi, (m)}_r)) \, \d W_r,
\end{equation}

\noindent where we set $\mathcal{L}(X^{s,\xi,(0)}_r) = \nu$, for any $r\in [0,T]$, for some probabilty measure $\nu \neq \mu$.

Under \textbf{(HE)} and \textbf{(HR)}, there exists a unique weak solution to the SDE \eqref{Xm1} for any given initial condition $(s, \xi)$. We also introduce the associated decoupled stochastic flow or characteristics given by the unique weak solution to the SDE with dynamics
 \begin{equation}\label{Xm2}
     X^{s,x,\mu,(m+1)}_t=x+\int_s^t b(r,X^{s,x,\mu,(m+1)}_r,\mathcal{L}(X^{s, \xi, (m)}_r)) \, \d r + \int_s^t \sigma(r,X^{s,x,\mu,(m+1)}_r,\mathcal{L}(X^{s, \xi, (m)}_r)) \, \d W_r.
 \end{equation}

It is clear that, by weak uniqueness, $X^{s,x,\mu,(m)}_t$ makes sense since the law of the random vector $X^{\xi,(m)}_t$ depends only on the initial condition $\xi$ through its law $\mu$. We then define the map
\begin{equation}\label{def:vm}
    \mathscr{V}^{(m)}:(t,x,\mu)\in[0,T]\times \mathbb{R}^d\times \mathcal{P}_2(\mathbb{R}^d)\mapsto \mathscr{V}^{(m)}(t,x,\mu)=\mathbb{E}\Big[\int_t^T b(s,X^{t,x,\mu, (m)}_s,\mathcal{L}(X^{t, \xi, (m)}_s)) \, \d s \Big].
\end{equation}

The proof of the following technical result follows from Theorem 3.8 in \cite{CHAUDRUDERAYNAL20221} (see also the proof of Theorem 3.6 in \cite{CHAUDRUDERAYNAL20211}) and is thus omitted.
\begin{lemma}\label{EstimUm}
    (On the regularity properties of $\mathscr{V}^{(m)}$) For any positive integer $m$, the following properties holds:
    \begin{itemize}
        \item [(i)]$\mathscr{V}^{(m)}$ is in $\mathcal{C}^{1,2,2}_f([0,T)\times \mathbb{R}^d\times \mathcal{P}_2(\mathbb{R}^d))$.
        \item [(ii)]There exists a constant $K=K(T, (\textbf{HE}),(\textbf{HR}))< \infty$, $T\mapsto K(T, (\textbf{HE}),(\textbf{HR}))$ being non-decreasing, such that for all positive integer $m$ and all $ (t,x,\mu, \upsilon, \upsilon')\in[0,T)\times\mathbb{R}^d\times\mathcal{P}_2(\mathbb{R}^d)\times(\mathbb{R}^d)^2$ and all $n=0,1,2:$
        $$
        \lvert \partial_\mu \mathscr{V}^{(m)}(t,x,\mu)(\upsilon)\rvert\leq K(T-t)^{\frac{1+\eta}{2}}, \quad \lvert \partial_x \partial_\mu \mathscr{V}^{(m)}(t,x,\mu)(\upsilon) \lvert + \lvert \partial^2_\mu \mathscr{V}^{(m)}(t,x,\mu)(\upsilon,\upsilon')\rvert\leq K(T-t)^{\frac{\eta}{2}}, 
        $$
        $$
        \lvert \partial_\upsilon \partial_\mu \mathscr{V}^{(m)}(t,x,\mu)(\upsilon)\rvert\leq K(T-t)^{\frac{\eta}{2}}, \quad \lvert\partial_t \mathscr{V}^{(m)}(t,x,\mu)(\upsilon)\rvert \leq K, \quad \lvert\partial^n_x \mathscr{V}^{(m)}(t,x,\mu)\rvert \leq K(T-t)^{\frac{2-n+\eta}{2}}.
        $$
        \item [(iii)]
   For all $ (t,x,\mu)\in[0,T)\times\mathbb{R}^d\times\mathcal{P}_2(\mathbb{R}^d), \quad \lim_{m\uparrow \infty} \mathscr{V}^{(m)}(t,x,\mu) = \mathscr{V}(t,x,\mu)$.
   \item [(iv)] $\left\{ (\partial_t+\mathcal{L}_t + \mathscr{L}_t) \mathscr{V}^{(m)}, m\geq1 \right\}$ converges to $(\partial_t+\mathcal{L}_t + \mathscr{L}_t) \mathscr{V} = b$,  locally uniformly along a subsequence, i.e. the convergence holds for all $(t,x,\mu)\in\mathcal{K}\subseteq[0,T) \times \mathbb{R}^d \times \mathcal{P}_2(\mathbb{R}^d)$, where $\mathcal{K}$ is compact for the canonical topology.
     \end{itemize}
\end{lemma}

\subsection{Proof of Theorem \ref{Thm:Conv:EulerMaru}}\label{proof:Thm:Conv:EulerMaru}

\noindent \emph{Step 1:} \\

Applying It\^o's rule to $\mathscr{V}^{(m)}(t,X^{i,n}_t,\mu^{N,n}_t)$, where $\mathscr{V}^{(m)}$ is defined by \eqref{def:vm}, we get
\begin{equation*}
\begin{aligned}
& \mathscr{V}^{(m)}(t,X^{i,n}_t,\mu^{N,n}_t) = \mathscr{V}^{(m)}(0,\xi^i,\mu^N_0)+\int_0^t(\partial_s+\mathcal{L}_s + \mathscr{L}_s)\mathscr{V}^{(m)}(s,X^{i,n}_s,\mu^{N,n}_s) \, \d s\\
&+\int_0^t \partial_x \mathscr{V}^{(m)}(s,X^{i,n}_s,\mu^{N,n}_s)\cdot \sigma(k_n(s),X^{i,n}_{k_n(s)},\mu^{N,n}_{k_n(s)}) \, \d W^{(i)}_s\\
&+\int_0^t(b(k_n(s),X^{i,n}_{k_n(s)},\mu^{N,n}_{k_n(s)})-b(s,X^{i,n}_s,\mu^{N,n}_s))\cdot \partial_x \mathscr{V}^{(m)}(s,X^{i,n}_s,\mu^{N,n}_s) \, \d s\\
&+\frac{1}{N}\sum_{l=1}^N\int_0^t (b(k_n(s),X^{i,n}_{k_n(s)},\mu^{N,n}_{k_n(s)})-b(s,X^{i,n}_s,\mu^{N,n}_s))\cdot \partial_\mu \mathscr{V}^{(m)}(s,X^{i,n}_s,\mu^{N,n}_s)(X^{l,n}_s) \, \d s\\
&+\frac{1}{2N}\sum_{l=1}^N\int_0^t \text{trace}\big([a(k_n(s),X^{i,n}_{k_n(s)},\mu^{N,n}_{k_n(s)})-a(s,X^{i,n}_s,\mu^{N,n}_s)]\partial_v\partial_\mu \mathscr{V}^{(m)}(s,X^{i,n}_s,\mu^{N,n}_s)(X^{l,n}_s)\big) \, \d s\\
&+\frac{1}{N}\sum_{l=1}^N\int_0^t \partial_\mu \mathscr{V}^{(m)}(s,X^{i,n}_s,\mu^{N,n}_s)(X^{l,n}_s) \cdot \sigma(k_n(s),X^{i,n}_{k_n(s)},\mu^{N,n}_{k_n(s)}) \, \d W^{(l)}_s\\
&+\frac{1}{2N^2}\sum_{l=1}^N\int_0^t \text{trace}(a(k_n(s),X^{i,n}_{k_n(s)},\mu^{N,n}_{k_n(s)})\partial^2_\mu \mathscr{V}^{(m)}(s,X^{i,n}_s,\mu^{N,n}_s)(X^{l,n}_s,X^{l,n}_s)) \, \d s
\end{aligned}
\end{equation*}
\normalsize

\noindent and similarly
\begin{equation*}
    \begin{aligned}
        \mathscr{V}^{(m)}(t,\bar{X}^i_t,\mathcal{L}(X_t^{\xi})) & = \mathscr{V}^{(m)}(0,\xi^i,\mu_0)+\int_0^t\partial_x \mathscr{V}^{(m)}(s,\bar{X}^i_s,\mathcal{L}(X_s^{\xi}))\cdot\sigma(s,\bar{X}^i_s,\mathcal{L}(X_s^{\xi})) \, \d W^{(i)}_s\\
        &+\int_0^t(\partial_s+\mathcal{L}_s + \mathscr{L}_s) \mathscr{V}^{(m)}(s,\bar{X}^i_s,\mathcal{L}(X_s^{\xi})) \, \d s.
    \end{aligned}
\end{equation*}

Hence, recalling the dynamics of $(\bar X^i_t)_{t\geq 0}$ and $(X^{i,n}_t)_{t\geq 0}$, and noticing that
\begin{align*}
    \bar X^i_t-X^{i,n}_t&=\big(\bar{X}^i_t- \mathscr{V}^{(m)}(t,\bar{X}^i_t,\mathcal{L}(X_t^{\xi}))\big)-\big(X^{i,n}_t-\mathscr{V}^{(m)}(t,X^{i,n}_t,\mu^{N,n}_t)\big)\\
 &+\mathscr{V}^{(m)}(t,\bar{X}^i_t,\mathcal{L}(X_t^{\xi}))- \mathscr{V}^{(m)}(t,X^{i,n}_t,\mu^{N,n}_t)
\end{align*}
we obtain, 
\small
\begin{equation}\label{DiffTwoProc}
    \begin{aligned}
        &\bar{X}^i_t-X^{i,n}_t = \mathscr{V}^{(m)}(0,\xi^i,\mu_0^N)-\mathscr{V}^{(m)}(0,\xi^i,\mu_0)+ \mathscr{V}^{(m)}(t,\bar{X}^i_t,\mathcal{L}(X_t^{\xi}))- \mathscr{V}^{(m)}(t,X^{i,n}_t,\mu^{N,n}_t)\\
    &+\int_0^t \Bigg[\big((\partial_x \mathscr{V}^{(m)}-\mathbbm{1})\sigma\big)(s,{X}^{i,n}_s,\mu^{N,n}_s)-\big((\partial_x \mathscr{V}^{(m)}-\mathbbm{1})\sigma\big)(s,\bar{X}^i_s, \mathcal{L}(X_s^{\xi}))\Bigg]\d W^{(i)}_s\\  &+\int_0^t\Bigg[((\partial_s+\mathcal{L}_s + \mathscr{L}_s)\mathscr{V}^{(m)} - b)(s,X^{i,n}_s,\mu^{N,n}_s)-((\partial_s+\mathcal{L}_s + \mathscr{L}_s)\mathscr{V}^{(m)}-b)(s,\bar{X}^i_s,\mathcal{L}(X_s^{\xi}))\Bigg]\d s\\
    &-\int_0^t (b(k_n(s),X^{i,n}_{k_n(s)},\mu^{N,n}_{k_n(s)})-b(s,X^{i,n}_s,\mu^{N,n}_s)) \, \d s-\int_0^t(\sigma(k_n(s),X^{i,n}_{k_n(s)},\mu^{N,n}_{k_n(s)})-\sigma(s,X^{i,n}_s,\mu^{N,n}_s)) \, \d W^{(i)}_s\\
&+\int_0^t (b(k_n(s),X^{i,n}_{k_n(s)},\mu^{N,n}_{k_n(s)})-b(s,X^{i,n}_s,\mu^{N,n}_s))\cdot \partial_x \mathscr{V}^{(m)}(s,X^{i,n}_s,\mu^{N,n}_s) \, \d s\\
&+\int_0^t \partial_x \mathscr{V}^{(m)}(s,X^{i,n}_s,\mu^{N,n}_s) (\sigma(k_n(s),X^{i,n}_{k_n(s)},\mu^{N,n}_{k_n(s)})-\sigma(s,X^{i,n}_s,\mu^{N,n}_s)) \, \d W^{(i)}_s\\
&+\frac{1}{N}\sum_{l=1}^N\int_0^t (b(k_n(s), X^{i,n}_{k_n(s)},\mu^{N,n}_{k_n(s)})-b(s, X^{i,n}_s,\mu^{N,n}_s))\cdot \partial_\mu \mathscr{V}^{(m)}(s,X^{i,n}_s,\mu^{N,n}_s)(X^{l,n}_s) \, \d s\\
&+\frac{1}{N}\sum_{l=1}^N\int_0^t \text{trace}\big([a(k_n(s),X^{i,n}_{k_n(s)},\mu^{N,n}_{k_n(s)})-a(s,X^{i,n}_s,\mu^{N,n}_s)]\partial_\upsilon\partial_\mu \mathscr{V}^{(m)}(s,X^{i,n}_s,\mu^{N,n}_s)(X^{l,n}_s)\big) \, \d s\\
&+\frac{1}{N}\sum_{l=1}^N\int_0^t \partial_\mu \mathscr{V}^{(m)}(s,X^{i,n}_s,\mu^{N,n}_s)(X^{l,n}_s)\sigma(k_n(s),X^{i,n}_{k_n(s)},\mu^{N,n}_{k_n(s)}) \, \d W^{(l)}_s\\
&+\frac{1}{2N^2}\sum_{l=1}^N\int_0^t \text{trace}(a(k_n(s),X^{i,n}_{k_n(s)},\mu^{N,n}_{k_n(s)})\partial^2_\mu \mathscr{V}^{(m)}(s,X^{i,n}_s,\mu^{N,n}_s)(X^{l,n}_s,X^{l,n}_s)) \, \d s.
    \end{aligned}
\end{equation}
\normalsize

 \noindent \emph{Step 2:} \\
 
We now establish appropriate upper-bounds for each term appearing in the right-hand side of the above identity.

It follows from Lemma \ref{EstimUm}(ii) that the map $\mu \mapsto \mathscr{V}^{(m)}(t, x, \mu)$ (resp. $x\mapsto \mathscr{V}^{(m)}(t, x, \mu)$) is Lipschitz continuous with respect to the product distance $\mathcal{W}_1$ distance (resp. with respect to the Euclidean distance) with Lipschitz constant bounded by $K (T-t)^{\frac{1+\eta}{2}}$ uniformly in $m$ and with respect to the variables $t$ and $x$ (resp. $t$ and $\mu$). Hence, there exists a constant $C_T=C(T, (\textbf{HR}), (\textbf{HE}))>0$ with $\lim_{T\downarrow 0} C_T=0$ such that 
$$
\E[\lvert \mathscr{V}^{(m)}(0,\xi^1,\mu^N_0)-\mathscr{V}^{(m)}(0,\xi^1,\mu_0)\rvert^2]\leq C_T\E[\mathcal{W}_1(\mu^N_0,\mu_0)^2],
$$ 

\noindent and 
$$
\E[\lvert \mathscr{V}^{(m)}(t,\bar{X}^i_t, \mathcal{L}(X_t^{\xi}))-\mathscr{V}^{(m)}(t,X^{i,n}_t,\mu^{N,n}_t)\rvert^2]\leq C_T(\E[\lvert\bar{X}^i_t-X^{i,n}_t\rvert^2]+\E[\mathcal{W}_1(\mathcal{L}(X_t^{\xi}),\mu^{N,n}_t)^2]).
$$

Similarly, it follows from the Burkholder-Davis-Gundy inequality, Lemma \ref{EstimUm}(ii) and the uniform Lipschitz regularity of $(x,\mu)\mapsto\sigma(t,x,\mu)$ that
\begin{align*}
\mathbb{E}\Big[ \Big(\int_0^t \big[&\big((\partial_x \mathscr{V}^{(m)}  -\mathbbm{1})\sigma\big)(s,{X}^{i,n}_s,\mu^{N,n}_s)  -\big((\partial_x \mathscr{V}^{(m)}-\mathbbm{1})\sigma\big)(s,\bar{X}^i_s,\mathcal{L}(X_s^{\xi}))\big]\d W^{(i)}_s \Big)^2\Big] \\
& \leq C \int_0^t \E\Big[\Big(\big((\partial_x \mathscr{V}^{(m)} -\mathbbm{1})\sigma\big)(s,{X}^{i,n}_s,\mu^{N,n}_s)-\big((\partial_x \mathscr{V}^{(m)}-\mathbbm{1})\sigma\big)(s,\bar{X}^i_s,\mathcal{L}(X_s^{\xi}))\Big)^2\Big] \, \d s \\
& \leq C\Big(\int_0^t \E[\lvert X^{i,n}_s-\bar{X}^i_s\rvert^2] \, \d s + \int_0^t \E[\mathcal{W}_1(\mu^{N,n}_s,\mathcal{L}(X_s^{\xi}))^2] \, \d s \Big).
\end{align*}

The fifth and sixth terms are handled using the H\"older regularity of the maps $(t, x, m) \mapsto b(t, x, m)$ and $t\mapsto \sigma(t, x, m)$ (uniformly in $x$ and $m$) as well as the uniform Lipschitz regularity of $(x, m)\mapsto \sigma(t,x,m)$. Hence, for some constant $C<\infty$, we get
\begin{align*}
\E[\lvert & b(k_n(s),X^{i,n}_{k_n(s)},\mu^{N,n}_{k_n(s)})-b(s,X^{i,n}_s,\mu^{N,n}_s)\rvert^2 ] +\E[\lvert \sigma(k_n(s),X^{i,n}_{k_n(s)},\mu^{N,n}_{k_n(s)})-\sigma(s,X^{i,n}_s,\mu^{N,n}_s)\rvert^2 ] \\
& \quad \leq C\big((s-k_n(s))^\eta+\E[\lvert X^{i,n}_{k_n(s)}-X^{i,n}_{s}\rvert^{2\eta}]+\E[\mathcal{W}_2(\mu^{N,n}_{k_n(s)},\mu^{N,n}_{s})^{2\eta}]\big),
\end{align*} 
\noindent so that
\begin{align*}
\E\Big[\Big| \int_0^t & (b(k_n(s),X^{i,n}_{k_n(s)},\mu^{N,n}_{k_n(s)})-b(s,X^{i,n}_s,\mu^{N,n}_s)) \, \d s \Big|^2 \Big] \\
& \quad +\E\Big[\Big| \int_0^t (\sigma(k_n(s),X^{i,n}_{k_n(s)},\mu^{N,n}_{k_n(s)})-\sigma(s,X^{i,n}_s,\mu^{N,n}_s)) \, \d W^{(i)}_s \Big|^2 \Big] \\
& \quad \quad \leq C \int_0^t \Big((s-k_n(s))^\eta+\E[\lvert X^{i,n}_{k_n(s)}-X^{i,n}_{s}\rvert^{2\eta}]+\E[\mathcal{W}_2(\mu^{N,n}_{k_n(s)},\mu^{N,n}_{s})^{2\eta}]\Big) \, \d s\\
& \quad \quad \leq C \int_0^t \Big((s-k_n(s))^\eta+\E[\lvert X^{i,n}_{k_n(s)}-X^{i,n}_{s}\rvert^{2}]^{\eta}\Big) \, \d s\\
& \quad \quad\leq C h^{\eta}. 
\end{align*}

The seventh, eighth, ninth and tenth terms are handled similarly using again Lemma \ref{EstimUm}(ii). Hence, 
\begin{align*}
& \E\Big[\Big|  \int_0^t  (b(k_n(s),X^{i,n}_{k_n(s)},\mu^{N,n}_{k_n(s)})-b(s,X^{i,n}_s,\mu^{N,n}_s))\cdot \partial_x \mathscr{V}^{(m)}(s,X^{i,n}_s,\mu^{N,n}_s) \, \d s \Big|^2 \Big] \\
&  +\E\Big[\Big| \int_0^t \partial_x \mathscr{V}^{(m)}(s,X^{i,n}_s,\mu^{N,n}_s)(\sigma(k_n(s),X^{i,n}_{k_n(s)},\mu^{N,n}_{k_n(s)})-\sigma(s,X^{i,n}_s,\mu^{N,n}_s)) \, \d W^{(i)}_s \Big|^2 \Big] \\
&  + \E\Big[\Big| \frac{1}{N}\sum_{l=1}^N\int_0^t (b(k_n(s), X^{i,n}_{k_n(s)},\mu^{N,n}_{k_n(s)})-b(s, X^{i,n}_s,\mu^{N,n}_s))\cdot \partial_\mu \mathscr{V}^{(m)}(s,X^{i,n}_s,\mu^{N,n}_s)(X^{l,n}_s) \, \d s \Big|^2 \Big]\\
&  + \E\Big[ \Big| \frac{1}{N}\sum_{l=1}^N\int_0^t \text{trace}\big([a(k_n(s),X^{i,n}_{k_n(s)},\mu^{N,n}_{k_n(s)})-a(s,X^{i,n}_s,\mu^{N,n}_s)]\\
& \quad \quad \partial_\upsilon\partial_\mu \mathscr{V}^{(m)}(s,X^{i,n}_s,\mu^{N,n}_s)(X^{l,n}_s)\big) \, \d s \Big|^2 \Big]\\
& \quad \quad \quad \leq C h^{\eta}. 
\end{align*}

We bound the eleventh term in the right-hand side of \eqref{DiffTwoProc} using the Burkholder-Davis-Gundy inequality, the boundedness of the diffusion matrix and again Lemma \ref{EstimUm}(ii). Hence, there exists a constant $C<\infty$ such that for all
\begin{align*}
\E\Big[\Big(\sup_{t\leq T}  \frac{1}{N}\sum_{l=1}^N & \int_0^t \partial_\mu \mathscr{V}^{(m)}(s,X^{i,n}_s,\mu^{N,n}_s)(X^{l,n}_s)\sigma(k_n(s),X^{i,n}_{k_n(s)},\mu^{N,n}_{k_n(s)})\d W^{(l)}_s \Big)^2\Big]\\
& \leq \frac{C}{N^2}\sum_{l=1}^N\int_0^T \E[\lvert\partial_\mu \mathscr{V}^{(m)}(s,X^{i,n}_s,\mu^{N,n}_s)(X^{l,n}_s)\sigma(k_n(s),X^{i,n}_{k_n(s)},\mu^{N,n}_{k_n(s)})\rvert^2] \, \d s\\
& \leq \frac{C}{N}.
\end{align*}

The last term in the right-hand side of \eqref{DiffTwoProc} is handled using the boundedness of the diffusion matrix $a=\sigma\sigma^{t}$ and Lemma \ref{EstimUm}(ii). Hence, there exists a constant $C<\infty$ such that 
$$
\sup_{0\leq t\leq T}\Big| \frac{1}{2N^2}\sum_{l=1}^N\int_0^t \text{trace}(a(k_n(s),X^{i,n}_{k_n(s)},\mu^{N,n}_{k_n(s)})\partial^2_\mu \mathscr{V}^{(m)}(s,X^{i,n}_s,\mu^{N,n}_s)(X^{l,n}_s,X^{l,n}_s)) \, \d s \Big|\leq \frac{C}{N}.
$$

Gathering the previous terms and using Gr\"onwall's inequality, we obtain
\begin{equation}\label{ineq:square:before:limit:m}
\begin{aligned}
\mathbb{E}[|\bar{X}^i_t - X^{i, n}_t|^2]& \leq C_T( \E[\mathcal{W}_1(\mu^N_0,\mu_0)^2] + \E[\lvert\bar{X}^i_t-X^{i,n}_t\rvert^2] + \E[\mathcal{W}_1(\mathcal{L}(X_t^{\xi}),\mu^{N,n}_t)^2] \\
& + \int_0^t \E[\mathcal{W}_1(\mathcal{L}(X_s^{\xi}),\mu^{N,n}_s)^2] \, \d s ) + C (h^\eta + \frac1N)\\
& \quad + C \E\Big[\Big(\int_0^T \lvert ((\partial_s+\mathcal{L}_s + \mathscr{L}_s)\mathscr{V}^{(m)}-b)(s,X^{i,n}_s,\mu^{N,n}_s)\rvert \d s\Big)^2\Big]\\
& \quad + C \E\Big[\Big(\int_0^T \lvert ((\partial_s+\mathcal{L}_s+ \mathscr{L}_s)\mathscr{V}^{(m)}-b)(s,\bar{X}^i_s,\mathcal{L}(X_s^{\xi}))\rvert \d s \Big)^2\Big],
\end{aligned}
\end{equation}

\noindent for some constants $C,\, C_T< \infty$ such that $\lim_{T\downarrow 0} C_T=0$. 

Our aim now is to pass to the limit as $m\uparrow \infty$ in the previous inequality and show that the last two terms, denoted respectively by ${\rm I}^{m}$ and ${\rm J}^m$, vanish.

We deal with ${\rm I}^m$ by using the following decomposition
\begin{align*}
{\rm I}^m & = \E\Bigg[\Big(\int_0^T |((\partial_s+\mathcal{L}_{s}+\mathscr{L}_s) \mathscr{V}^{(m)} -b)(s, X_s^{i, n}, \mu^{N,n}_{s})| \textbf{1}_{|\mathbf{X}_{s}^{N,n}|\leq R} \, \d s \Big)^2\Bigg] \\
& \quad + \E\Bigg[ \Big( \int_0^T |( (\partial_s+\mathcal{L}_{s} + \mathscr{L}_s)\mathscr{V}^{(m)}-b)(s, X^{i, n}_s, \mu^{N,n}_{s})| \textbf{1}_{|\mathbf{X}_{s_1}^{N,n}| >  R}\, \d s \Big)^2 \Bigg] \\
& =: {\rm I}^{m, R}_{1} + {\rm I}^{m, R}_{2}, 
\end{align*}

\noindent where $R>0$ is a fixed parameter, recalling that $\mathbf{X}^{N, n}_t = (X^{1, n}_t, \cdots, X^{N, n}_t)$. 

According to Lemma \ref{EstimUm}(iv), for any $R>0$, $(\partial_s + \mathcal{L}_s + \mathscr{L}_{s})  \mathscr{V}^{(m)}(s,X^{i, n}_s, \mu^{N,n}_{s}) \textbf{1}_{|\mathbf{X}^{N, n}_{s}|\leq R }$ converges $a.s.$ to $(\partial_s + \mathcal{L}_s + \mathscr{L}_{s})  \mathscr{V}(s, X_s^{i, n}, \mu^{N,n}_{s}) \textbf{1}_{|\mathbf{X}^{N, n}_{s}|\leq R } = b(s, X_s^{i, n},\mu^{N,n}_{s}) \textbf{1}_{|\mathbf{X}^{N, n}_{s}|\leq R }$, along a subsequence, which in turn using the boundedness of the coefficients $b_i$ and $a_{i,j}$, the estimates of Lemma \ref{EstimUm}(ii) and the dominated convergence theorem yield
\begin{align*}
\lim_{m\uparrow \infty} {\rm I}^{m, R}_{1} & = 0.
\end{align*}

 From Lemma \ref{EstimUm}(ii) and the boundedness of $b$, we get
\begin{align*}
   \sup_{m\geq 1}{\rm I}_2^{m, R} \leq C \mathbb{P}(\sup_{0\leq s\leq T}|\mathbf{X}^{N, n}_{s}|> R)
\end{align*}

\noindent which clearly yields 
\begin{equation}\label{esti:second:term:part_temp}
\lim\sup_{R \uparrow \infty} \sup_{m\geq 1}{\rm I}_2^{m, R} = 0.
\end{equation}

Hence, coming back to the decomposition of ${\rm I}^m$ and passing to the limit as $m \uparrow \infty$ and then as $R \uparrow \infty$, we conclude that $\lim_m {\rm I}^m = 0$. Omiting some technical details, we similarly prove that $\lim_m {\rm J}^m = 0$.

We now pass to the limit as $m\uparrow \infty$ in \eqref{ineq:square:before:limit:m} and then sum over $i=1, \cdots, N$. Hence, 
\begin{equation}\label{ineq:square:after:limit:m}
\begin{aligned}
\mathbb{E}\Big[\frac{1}{N}\sum_{i=1}^{N}|\bar{X}^i_t - X^{i, n}_t|^2 \Big]& \leq C_T\Big( \E[\mathcal{W}_1(\mu^N_0,\mu_0)^2] + \E\Big[\frac{1}{N}\sum_{i=1}^{N}\lvert\bar{X}^i_t-X^{i,n}_t\rvert^2\Big] \\
& \quad +\E[\mathcal{W}_1(\mathcal{L}(X_t^{\xi}),\mu^{N,n}_t)^2] \Big)\\
& \quad + C\int_0^t \E[\mathcal{W}_1(\mathcal{L}(X_s^{\xi}),\mu^{N,n}_s)^2] \, \d s + C (h^\eta + \frac1N).
\end{aligned}
\end{equation}

We now introduce $\bar{\mu}^N_t=\frac{1}{N}\sum_{i=1}^N\delta_{\bar{X}^i_t}$ the empirical measure of the i.i.d random variables $(\bar{X}^i_t)_{1\leq i\leq N}$ and use the triangle inequality to deduce that
\begin{equation}\label{EstimDistW2mutmuNnt}
\begin{aligned}
    \mathcal{W}_1(\mathcal{L}(X_t^{\xi}),\mu^{N,n}_t)^2& \leq 2\mathcal{W}_1(\mathcal{L}(X_t^{\xi}),\bar{\mu}^N_t)^2+2\mathcal{W}_1(\bar{\mu}^N_t,\mu^{N,n}_t)^2\\
    & \leq 2\mathcal{W}_1(\mathcal{L}(X_t^{\xi}),\bar{\mu}^N_t)^2+\frac{2}{N}\sum_{i=1}^N\lvert \bar{X}^i_t-X^{i,n}_t\rvert^2.
\end{aligned}
\end{equation}

We plug the previous inequality into \eqref{ineq:square:after:limit:m}, use again Gr\"onwall's inequality and choose $T$ small enough such that $C_T\leq\frac{1}{6}$. We obtain
\begin{align*}
   \E\Big[\frac{1}{N}\sum_{i=1}^N\lvert\bar{X}^i_t-X^{i,n}_t\rvert^2 \Big]& \leq C_T \Big( \E[\mathcal{W}_1(\mu^N_0,\mu_0)^2]+\sup_{0\leq t\leq T}\E[\mathcal{W}_1(\mathcal{L}(X_t^{\xi}),\bar{\mu}^{N}_t)^2] \Big)\\
   & \quad +C (h^{\eta} + \frac{1}{N}).
\end{align*}

Finally, the strong well-posedness of the SDE \eqref{SDEbarXi} as well as the exchangeability of $(\xi^i,W^{(i)})$ imply that $(\bar{X}^i,X^{i,n})_{1\leq i\leq N}$ are exchangeable in law so that
$$
\frac{1}{N}\E\Big[\sum_{i=1}^N\lvert\bar{X}^i_t-X^{i,n}_t\rvert^2\Big]=\E[\lvert\bar{X}^1_t-X^{1,n}_t\rvert^2],
$$

\noindent which in turn implies
\begin{align*}
    \max_{1\leq i\leq N}\sup_{0\leq t\leq T}& \E[\lvert\bar{X}^i_t-X^{i,n}_t\rvert^2] \\ 
        &\leq C_T \Big( \E[\mathcal{W}_1(\mu^N_0,\mu_0)^2]+\sup_{0\leq t\leq T}\E[\mathcal{W}_1(\mathcal{L}(X_t^{\xi}),\bar{\mu}^{N}_t)^2] \Big)+C(h^{\eta} + \frac{1}{N}).
\end{align*}

Now, it follows from the standard inequality $\mathcal{W}_1 \leq \mathcal{W}_2$ and Theorem 1 in \cite{fournier2013rate} that 
\begin{equation}\label{fournier:guillin:consequence}
\E[\mathcal{W}_1(\mu^N_0,\mu_0)^2]+\sup_{0\leq t\leq T}\E[\mathcal{W}_1(\mathcal{L}(X_t^{\xi}),\bar{\mu}^{N}_t)^2]\leq C\varepsilon_N,
\end{equation}

\noindent for some constant $C$ that depends only upon $d$, $q$ and $M_q(\mu)$ and where $\varepsilon_N$ is defined in \eqref{def:epsilonN}. Hence, we conclude that there exists a constant $C=C((\textbf{HR}), (\textbf{HE}), T, d, q, M_q(\mu))>0$  such that 
\begin{equation*}
    \max_{1\leq i\leq N}\sup_{0\leq t\leq T}\E[\lvert\bar{X}^1_t-X^{1,n}_t\rvert^2]\leq C(\varepsilon_N+h^\eta).
\end{equation*}

It follows from the triangle inequality and the previous inequality that 
\begin{align*}
    \sup_{0\leq t \leq T} \mathbb{E}[\mathcal{W}_2(\mathcal{L}(X_t^{\xi}), \mu_t^{N, n})^2] &\leq 2 (\mathbb{E}[\mathcal{W}_2(\mathcal{L}(X_t^{\xi}), \bar{\mu}_t^{N})^2] + \max_{1\leq i \leq N}\sup_{0\leq t \leq T}\mathbb{E}[|\bar{X}^i_t-X^{i, n}_t|^2])\\
        & \leq C(\varepsilon_N+h^\eta).    
\end{align*}

The first inequality \eqref{strong:error:sup:outisde} of Theorem \ref{Thm:Conv:EulerMaru} is eventually obtained by combining the two previous inequalities for $T$ small enough.

One may then extend the above estimate to an arbitrary finite time horizon $T$ by considering a partition of the time interval $[0,T]$ with a sufficiently small time mesh and repeating the above argument, observing that the estimates of Lemma \ref{lemma:estimate:deriv:um} are uniform in $x$, $\mu$, $\upsilon$ and $\upsilon'$. \\

\noindent \emph{Step 3:} \\

In order to prove the second inequality \eqref{strong:error:sup:inside} of Theorem \ref{Thm:Conv:EulerMaru}, we come back to \eqref{DiffTwoProc} and follow similar lines of reasonings but taking first the square of the norm, then the supremum in $t\in [0,T]$. Hence, we obtain, thanks to the above estimate, for $T$ small enough,
\begin{equation*}
    \begin{split}
    \max_{1\leq i\leq N}\mathbb{E}[\sup_{0\leq t\leq T}\lvert\bar{X}^1_t-X^{1,n}_t\rvert^2] &\leq C_T \Big( \mathbb{E}[\mathcal{W}_1(\mu^N_0,\mu_0)^2]+\mathbb{E}[\sup_{0\leq t\leq T}\mathcal{W}_1(\mathcal{L}(X_t^{\xi}),\bar{\mu}^{N}_t)^2] \\
    & \quad + \int_0^T \mathbb{E}[\mathcal{W}_1(\bar{\mu}^{N}_s,\mathcal{L}(X_s^{\xi}))^2] \, \d s \Big) +C \Big(h^\eta + \frac{1}{N} \Big).   
    \end{split}
\end{equation*}

The second term is handled using Lemma 1 \cite{ForBackwardSDENorCons}. We obtain
$$
\mathbb{E}[\sup_{0\leq t\leq T}\mathcal{W}_1(\mathcal{L}(X_t^{\xi}),\bar{\mu}^{N}_t)^2]\leq \mathbb{E}[\sup_{0\leq t\leq T}\mathcal{W}_2(\mathcal{L}(X_t^{\xi}),\bar{\mu}^{N}_t)^2]\leq C\sqrt{\varepsilon_N}.
$$

The first and third terms are handled as previously using \eqref{fournier:guillin:consequence}. Hence, we obtain 
$$
 \max_{1\leq i\leq N}\mathbb{E}[\sup_{0\leq t\leq T}\lvert\bar{X}^1_t-X^{1,n}_t\rvert^2]\leq C(\sqrt{\varepsilon_N}+h^\eta).
$$
 
It follows from the triangle inequality, the exchangeability in law of $(\bar{X}^i, X^{i, n})_{1\leq i \leq N}$ and the two previous inequalities that
\begin{equation}
    \begin{split}
       \mathbb{E}[\sup_{0\leq t\leq T} \mathcal{W}_2(\mathcal{L}(X_t^{\xi}),\mu^{N,n}_t)^2]&\leq 2\mathbb{E}[\sup_{0\leq t\leq T} \mathcal{W}_2(\mathcal{L}(X_t^{\xi}),\bar{\mu}^N_t)^2]+{2}\mathbb{E}[\sup_{0\leq t\leq T} \lvert \bar{X}^1_t-X^{1,n}_t\rvert^2]\\
       &\leq C(\sqrt{\varepsilon_N}+h^\eta).
    \end{split}
\end{equation}

The proof of \eqref{strong:error:sup:inside} is now complete.

\subsection{Auxiliary results for the proofs of Theorem \ref{Thm:Conv:EulerMaru:semigroup} and \ref{thm:error:upper:bound:smooth:case}}\label{AuxResMaruyama2}
The proofs of Theorems \ref{Thm:Conv:EulerMaru:semigroup} and \ref{thm:error:upper:bound:smooth:case} hinge upon the existence of a unique solution $\mathscr{U}:[0,t]\times \mathcal{P}_2(\mathbb{R}^d) \rightarrow \mathbb{R}$ to the backward Kolmogorov PDE:
\begin{equation}\label{Backward:Kolmorgorov:PDE:semigroup}
\begin{cases}
(\partial_t + \mathscr{L}_s)\mathscr{U}(s,\mu)&  = 0, \quad (s, \mu) \in [0,t) \times \mathcal{P}_2(\mathbb{R}^d), \\ 
\quad \mathscr{U}(t, \mu) & = \Phi(\mu), \quad \mu \in \mathcal{P}_2(\mathbb{R}^d),
\end{cases}
\end{equation}

\noindent where $\mathscr{L}_s$ is the operator given by \eqref{Lt}. Two different settings are here considered. In Theorem \ref{Thm:Conv:EulerMaru:semigroup}, the coefficients $b_i$, $a_{i, j}$ and the terminal condition $\Phi$ are irregular so that we will work with a regularized version $\mathscr{U}^{(m)}$ of the solution. In the smooth setting of Theorem \ref{thm:error:upper:bound:smooth:case}, the coefficients and terminal condition are smooth so that we will work directly with the unique solution $\mathscr{U}$ to the above PDE. 

\subsubsection{Irregular setting}\label{subsubsec:irregularsetting}

Under \textbf{(HE)} and \textbf{(HR)}, it follows from Theorem 3.8 \cite{CHAUDRUDERAYNAL20211} that for any $\Phi \in \mathscr{C}^{2,\alpha}(\mathcal{P}_2(\mathbb{R}^d), \mathbb{R}^d)$ there exists a unique solution $\mathscr{U} \in \mathcal{C}^{1,  2}([0,t)\times \mathcal{P}_2(\mathbb{R}^d)) \cap \mathcal{C}([0,t]\times \mathcal{P}_2(\mathbb{R}^d))$ to the above PDE which satisfies
$$
\mathscr{U}(s, \mu) = \Phi(\mathcal{L}({X}^{s, \xi}_t)), \quad (s, \mu) \in[0,t] \times \mathcal{P}_2(\mathbb{R}^d).
$$

Importantly, according to \cite{CHAUDRUDERAYNAL20211} and \cite{CHAUDRUDERAYNAL20221}, the existence of the second order Wasserstein derivative $\partial_\mu^2 \mathscr{U}(s,\mu)$ is not guaranteed under the sole assumption \textbf{(HE)} and \textbf{(HR)}. Hence, instead of working directly with $\mathscr{U}$, we will work with the regularized version $(\mathscr{U}^{(m)})_{m\geq1}$ defined by
$$
\mathscr{U}^{(m)}(s, \mu) = \Phi(\mathcal{L}({X}^{s, \xi, (m)}_t)), \quad (s, \mu) \in[0,t] \times \mathcal{P}_2(\mathbb{R}^d),
$$

\noindent recalling that $(X^{s, \xi, (m)}_t)_{s \leq t \leq T}$ is given by the unique weak solution to \eqref{Xm1}. It follows from Proposition 4.1 \cite{CHAUDRUDERAYNAL20211} that the density function $p_m(\mu, s, t, x, z)$ of the random vector $X_t^{s, x, \mu, (m)}$ exists for any $0\leq s < t \leq T$, recalling that $(X^{s, x, \mu, (m)}_t)_{s \leq t \leq T}$ is given by the unique weak solution to \eqref{Xm2}. We now recall some important Gaussian estimates on the derivatives of $p_m(\mu, s, t, z)$ directly taken from \cite{CHAUDRUDERAYNAL20221, CHAUDRUDERAYNAL20211}.

\begin{Lemma}\label{lemma:gaussian:bounds:pm} Assume that (\textbf{HE}) and (\textbf{HR}) hold. Then, for any $T>0$, any $(t, z) \in (0,T] \times \mathbb{R}^d$ and any integer $m$, the following properties hold: 

    \begin{itemize}
        \item The map $[0,t) \times \mathbb{R}^d \times \mathcal{P}_2(\mathbb{R}^d) \ni (s, x, \mu)\mapsto p_m(\mu, s, t, x, z)$ is in $\mathcal{C}^{1, 2, 2}_f([0,t)\times \mathbb{R}^d \times \mathcal{P}_2(\mathbb{R}^d))$.
        \item There exist positive constants $C:=C(T,(\textbf{HR}), (\textbf{HE}))$, $T\mapsto C(T,(\textbf{HR}), (\textbf{HE}))$ being non-decreasing, $c:=c(\lambda)$ such that for any $(s, x, \mu) \in [0,t) \times \mathbb{R}^d \times \mathcal{P}_2(\mathbb{R}^d)$ and any positive integer $m$
        \begin{align*}
            |\partial_x^{\ell} p_m(\mu, s, t, x, z)| & \leq \frac{C}{(t-s)^{\frac{\ell}{2}}} \, g(c(t-s), z-x), \quad \ell=0, 1, 2, \\
            |\partial^{\ell}_\mu p_m(\mu, s, t, x, z)|& \leq \frac{C}{(t-s)^{\frac{\ell-\eta}{2}}} \, g(c(t-s), z-x), \quad \ell=1, 2, \\
            |\partial_\upsilon \partial_\mu p_m(\mu, s, t, x, z)|& \leq \frac{C}{(t-s)^{1-\frac{\eta}{2}}} \, g(c(t-s), z-x),\\
            |\partial_x \partial_\mu p_m(\mu, s, t, x, z)|& \leq \frac{C}{(t-s)^{1-\frac{\eta}{2}}} \,  g(c(t-s), z-x),\\
            | \partial_s p_m(\mu, s, t, x, z)|& \leq \frac{C}{t-s} \, g(c(t-s), z-x).
        \end{align*}  
    \end{itemize}
\end{Lemma}

The following estimates will be useful for the proof of our second main result.
\begin{Lemma}\label{lemma:estimate:deriv:um}
  Assume that (\textbf{HE}) and (\textbf{HR}) hold. Then, there exists a positive constant $C = C(T, (\textbf{HR})), (\textbf{HE})) < \infty$, $T\mapsto C(T, (\textbf{HR}), (\textbf{HE}))$ being non-decreasing, such that for all positive integer $m$, for all $0\leq s < t \leq T$ and all $(\mu, \upsilon, \upsilon')\in \mathcal{P}_2(\mathbb{R}^d)\times(\R^d)^2$, it holds
$$
\lvert\partial_\mu \mathscr{U}^{(m)}(s,\mu)(\upsilon)\rvert\leq C(t-s)^{\frac{-1+\alpha}{2}} (1+|\upsilon|+M_2(\mu))),
$$
$$
\lvert\partial_\upsilon\partial_\mu \mathscr{U}^{(m)}(s,\mu)(\upsilon)\rvert + \lvert\partial^2_\mu \mathscr{U}^{(m)}(s,\mu)(\upsilon, \upsilon')\rvert + \lvert\partial_t \mathscr{U}^{(m)}(s,\mu)\rvert \leq C(t-s)^{-1+\frac{\alpha}{2}} (1+|\upsilon|+|\upsilon'| + M_2(\mu)). 
$$
\end{Lemma}
\begin{proof}

\noindent \emph{Step 1: }  It follows from Proposition 5.1 \cite{CHAUDRUDERAYNAL20221} and Proposition 4.1 \cite{CHAUDRUDERAYNAL20211} combined with Proposition 2.3 \cite{CHAUDRUDERAYNAL20211} that $\mathscr{U}^{(m)} \in \mathcal{C}^{1, 2}_f([0,t) \times \mathcal{P}_2(\mathbb{R}^d))$ with derivatives that can be expressed as follows:
{\small
\begin{equation}\label{first:order:measure:derivative:um}
\begin{aligned}
    \partial^{i}_\mu & \mathscr{U}^{(m)}(s,\mu)(\upsilon)\\
    &=\int_{\mathbb{R}^d}  \Big( \frac{\delta \Phi}{\delta m}(\mathcal{L}(X^{(s,\xi, (m)}_t))(y) - \frac{\delta \Phi}{\delta m}(\mathcal{L}(X^{s,\xi, (m)}_t))(\upsilon)\Big) \partial_{x_i} p_m(\mu, s,t, \upsilon,y) \, \d y\\
    & + \int_{(\mathbb{R}^d)^2} \Big( \frac{\delta \Phi}{\delta m}(\mathcal{L}(X^{(s,\xi, (m)}_t))(y) -\frac{\delta \Phi}{\delta m}(\mathcal{L}(X^{s,\xi, (m)}_t))(x) \Big)\, \partial^{i}_\mu p_m(\mu,s,t,x,y)(\upsilon) \, \d y \, \mu(\d x),
\end{aligned}
\end{equation}
\begin{equation}\label{cross:order:measure:derivative:um}
\begin{aligned}
    \partial_{\upsilon_j} & \partial^{i}_\mu \mathscr{U}^{(m)}(s,\mu)(\upsilon)\\
    &= \int_{\mathbb{R}^d}  \Big( \frac{\delta \Phi}{\delta m}(\mathcal{L}(X^{(s,\xi, (m)}_t))(y) - \frac{\delta \Phi}{\delta m}(\mathcal{L}(X^{s,\xi, (m)}_t))(\upsilon)\Big) \partial^2_{x_i, x_j} p_m(\mu, s, t, \upsilon, y) \, \d y \\
    & +\int_{(\mathbb{R}^d)^2} \Big(\frac{\delta \Phi}{\delta m}(\mathcal{L}(X^{s,\xi, (m)}_t))(y)-\frac{\delta \Phi}{\delta m}(\mathcal{L}(X^{(s,\xi, (m)}_t))(x)\Big)\, \partial_{\upsilon_{j}}\partial^{i}_\mu p_m(\mu, s, t,x,y)(\upsilon) \, \d y \, \mu(\d x),\\
\end{aligned}
\end{equation}

\begin{equation}\label{second:order:mes:derivative:um}
    \begin{aligned}
    &\partial^{(i,j)}_\mu \mathscr{U}^{(m)}(s,\mu)(\upsilon,\upsilon')\\
    & =\int_{(\mathbb{R}^d)^2} \Big( \frac{\delta^2 \Phi}{\delta m^2}(\mathcal{L}(X^{s,\xi, (m)}_t))(y,y') -\frac{\delta^2 \Phi}{\delta m^2}(\mathcal{L}(X^{(s,\xi, (m)}_t))(y,\upsilon') \Big) \partial_{x_i} p_m(\mu, s, t,\upsilon,y) \\
    & \quad \quad \times \partial_{x_j} p_m(\mu, s, t ,\upsilon',y') \, \, \d y \, \d y'\\
    &  + \int_{(\mathbb{R}^d)^3} \Big(\frac{\delta^2 \Phi}{\delta m^2}(\mathcal{L}(X^{s,\xi, (m)}_t))(y,y')-\frac{\delta^2 \Phi}{\delta m^2}(\mathcal{L}(X^{s,\xi, (m)}_t))(y,x)\Big) \partial_{x_i} p_m(\mu,s ,t,\upsilon,y)\\
    & \quad \quad \times \partial^{j}_\mu p_m(\mu,s , t,x,y')(\upsilon') \, \, \d y \, \d y' \, \mu(\d x)\\
    & +\int_{(\mathbb{R}^d)^4} \Big( \frac{\delta^2 \Phi}{\delta m^2}(\mathcal{L}(X^{s,\xi, (m)}_t))(y, y') - \frac{\delta^2 \Phi}{\delta m^2}(\mathcal{L}(X^{s,\xi, (m)}_t))(y, x') \Big) \, \partial^{i}_\mu p_m(\mu, s, t,x,y)(\upsilon) \\
    & \quad \quad \times \partial^{j}_\mu p_m(\mu, s,t,x',y')(\upsilon') \, \d y \, \d y' \, \mu(\d x)\, \mu(\d x')\\
     & +\int_{(\mathbb{R}^d)^3} \Big( \frac{\delta^2 \Phi}{\delta m^2}(\mathcal{L}(X^{s,\xi, (m)}_t))(y, y') - \frac{\delta^2 \Phi}{\delta m^2}(\mathcal{L}(X^{s,\xi, (m)}_t))(y, \upsilon')\Big) \, \partial^{i}_\mu p_m(\mu, s, t,x,y)(\upsilon) \\
     & \quad \quad \times \partial_{x_j} p_m(\mu, s, t,\upsilon',y') \, \d y \, \d y' \, \mu(\d x)\\
     &  +\int_{\mathbb{R}^d} \Big( \frac{\delta \Phi}{\delta m} (\mathcal{L}(X^{s,\xi, (m)}_t))(y) - \frac{\delta \Phi}{\delta m} (\mathcal{L}(X^{s,\xi, (m)}_t))(\upsilon) \Big) \, \partial^{j}_\mu\partial_{x_i} p_m(\mu, s, t, \upsilon, y)(\upsilon') \, \d y\\
    &  +\int_{\mathbb{R}^d}  \Big( \frac{\delta \Phi}{\delta m}(\mathcal{L}(X^{s,\xi, (m)}_t))(y) -  \frac{\delta \Phi}{\delta m}(\mathcal{L}(X^{s,\xi, (m)}_t))(\upsilon') \Big) \,\partial_{x_j} \partial^{i}_\mu p_m(\mu,s , t,\upsilon',y)(\upsilon) \, \d y\\
    &  +\int_{(\mathbb{R}^d)^2}  \Big( \frac{\delta \Phi}{\delta m}(\mathcal{L}(X^{s,\xi, (m)}_t))(y) - \frac{\delta \Phi}{\delta m}(\mathcal{L}(X^{s,\xi, (m)}_t))(x) \Big) \, \partial^{i, j}_\mu p_m(\mu, s,t,x,y)(\upsilon, \upsilon') \, \d y \, \mu(\d x)
\end{aligned}
\end{equation}
}
\noindent and {\small
\begin{equation}\label{time:derivative:um}
\partial_t \mathscr{U}^{(m)}(s, \mu) = \int_{\mathbb{R}^d} \Big(\frac{\delta \Phi}{\delta m}(\mathcal{L}(X^{s,\xi, (m)}_t))(y) - \frac{\delta \Phi}{\delta m}(\mathcal{L}(X^{s,\xi, (m)}_t))(x)\Big) \partial_s p_m(\mu, s, t, x,y) \, \d y \, \mu(\d x).
\end{equation}
}
\noindent \emph{Step 2: } Our aim now is to establish the bounds on the derivatives by using the local H\"older regularity of the flat derivatives of $\Phi$ combined with the uniform Gaussian bounds on the derivatives of $p_m$ recalled in Lemma \ref{lemma:gaussian:bounds:pm} and the space-time inequality \eqref{space:time:inequality}. 

In particular, we may break the first integral appearing in the right-hand side of \eqref{first:order:measure:derivative:um} into two parts ${\rm J}_1$ and ${\rm J}_2$ by dividing the domain of integration into two domains. In the first part ${\rm J}_1$, the $\d y$-integration is taken over a bounded domain $D$ that contains $\upsilon$ such that $|y-\upsilon| \geq 1$ if $y \notin D$. Using \eqref{local:holder:reg:first:order:linear:functional:deriv:phi}, that is, the $\alpha$-H\"older regularity of $ [\delta \Phi/\delta m](m)(.)$ on $D$, the Gaussian bound satisfied by $\partial_x p_m(\mu, s, t, x, y)$ of Lemma \ref{lemma:gaussian:bounds:pm} together with the space time inequality \eqref{space:time:inequality} and noting that $M_2(\mathcal{L}(X_t^{s, \xi, (m)})) \leq C(1+M_2(\mu))$, we get
$$
|{\rm J}_1| \leq C  \, (t-s)^{\frac{-1 + \alpha}{2}}  (1+ M_2(\mu)).
$$

As for ${\rm J}_2$, from \eqref{growth:condition:deriv:phi}, Lemma \ref{lemma:gaussian:bounds:pm} and again the space-time inequality \eqref{space:time:inequality}, we obtain
\begin{align*}
|{\rm J}_2| & \leq C  \,  \int_{|y-\upsilon| \geq 1} (t-s)^{-\frac{1-\eta}{2}} (1+ |y| + |\upsilon|+ M_2(\mathcal{L}(X^{s,\xi, (m)}_t))) \, g(c(t-s), y-\upsilon) \, \d y \,  \\
& \leq C  \,  (t-s)^{\frac{-1+\alpha}{2}} (1+ |\upsilon| + M_2(\mu)). 
\end{align*}

We deal with the second integral appearing in the right-hand side of \eqref{first:order:measure:derivative:um} using a similar decomposition. In particular, from \eqref{local:holder:reg:first:order:linear:functional:deriv:phi}, \eqref{growth:condition:deriv:phi}, Lemma \ref{lemma:gaussian:bounds:pm} and the space-time inequality \eqref{space:time:inequality}, we similarly get
\begin{align*}
\Big| \int_{(\mathbb{R}^d)^2} \Big( \frac{\delta \Phi}{\delta m}(\mathcal{L}(X^{s,\xi, (m)}_t))(y) & - \frac{\delta \Phi}{\delta m}(\mathcal{L}(X^{s,\xi, (m)}_t))(x) \Big)\, \partial^{i}_\mu p_m(\mu,s,t,x,y)(\upsilon) \, \d y \, \mu(\d x) \Big| \\
& \leq C (t-s)^{\frac{-1+\alpha}{2}} (1+M_2(\mu)).
\end{align*}

Gathering the three previous estimates, we obtain
\begin{equation}\label{first:deriv:mes:U}
\begin{aligned}
| \partial^{i}_\mu \mathscr{U}^{(m)}(s,\mu)(\upsilon)| 
& \leq C (t-s)^{\frac{-1+\alpha}{2}} (1 + |\upsilon| + M_2(\mu)).
\end{aligned}
\end{equation}

We use similar computations to deal with the other terms. In particular, from \eqref{cross:order:measure:derivative:um}, \eqref{second:order:mes:derivative:um}, \eqref{time:derivative:um}, \eqref{local:holder:reg:first:order:linear:functional:deriv:phi}, \eqref{growth:condition:deriv:phi},  Lemma \ref{lemma:gaussian:bounds:pm} and the space-time inequality \eqref{space:time:inequality}, we obtain
\begin{equation}\label{cross:deriv:mes:upsilon:U}
\begin{aligned}
    \lvert   \partial_{\upsilon_j} \partial^{i}_\mu \mathscr{U}^{(m)}(s,\mu)(\upsilon)\rvert & + \lvert   \partial^{(i,j)}_\mu \mathscr{U}^{(m)}(s,\mu)(\upsilon,\upsilon')\rvert +  \lvert   \partial_t  \mathscr{U}^{(m)}(s, \mu)\rvert \\
    &\leq C (t-s)^{-1+\frac{\alpha}{2}} (1 + |\upsilon|+ |\upsilon'| + M_2(\mu)).
\end{aligned}
\end{equation}

\end{proof}

\subsubsection{Regular setting}\label{subsubsec:regularsetting}

As previously mentioned, the strategy is similar to the one used in the proof of the previous results, namely, we rely on the well-posedness of the backward Kolmogorov PDE \eqref{Backward:Kolmorgorov:PDE:semigroup}. 

In the current smooth setting, we take advantage of the flow derivatives of the unique solution to the McKean-Vlasov SDE \eqref{SDE:MCKEAN} and its decoupling field to compute the derivatives of $\mathscr{U}(s,.)$. We let $(\widetilde{W}_{t})_{t\geq0}$ be a $\tilde{q}$-dimensional Brownian motion and consider the unique strong solution to the following stochastic differential equation with dynamics
\begin{equation}\label{MKV:auxiliary:for:proof}
\widetilde{X}_t^{s, \xi} = \xi + \int_s^t b(r, \widetilde{X}_r^{s, \xi}, \mathcal{L}({X}^{s, \xi}_r)) \, \d r + \int_s^t \tilde{\sigma}(r, \widetilde{X}_r^{s, \xi}, \mathcal{L}({X}^{s, \xi}_r)) \, \d \widetilde{W}_r, \quad s\leq t \leq T.
\end{equation}

 We also denote by $\widetilde{X}_t^{s, x, \mu}$ the unique solution to the decoupled SDE associated to \eqref{MKV:auxiliary:for:proof} starting from $x$ at time $s$. We now introduce the main object of interest, that is, the map 
 $$
 \widetilde{\mathscr{U}}(s, \mu) = \Phi(\mathcal{L}(\widetilde{X}^{s, \xi}_t)), \quad (s, \mu) \in [0,T] \times \mathcal{P}_2(\mathbb{R}^d).
 $$

 It follows from the well-posedness of the non-linear martingale problem associated to \eqref{MKV:auxiliary:for:proof} (which only depends on $\tilde{\sigma}$ through the diffusion matrix $a=\sigma \sigma^{\trans} = \tilde{\sigma} \tilde{\sigma}^{\trans}$) that $\mathcal{L}(\widetilde{X}^{s, \xi}_t) = \mathcal{L}({X}^{s, \xi}_t)$. Hence, the map $\mathscr{U}(s, \mu) = \Phi(\mathcal{L}({X}^{s, \xi}_t))$ satisfies for any $s \leq t \leq T$ and any $\mu \in \mathcal{P}_2(\mathbb{R}^d)$,
 $$
 \mathscr{U}(s, \mu) = \widetilde{\mathscr{U}}(s, \mu) = \Phi(\mathcal{L}(\widetilde{X}^{s, \xi}_t)). 
 $$


The next result, which is similar in nature to Lemmas 6.1 and 6.2 as well as Theorem 7.2 in \cite{buckdahn2017}, characterizes the regularity of $\mathscr{U}$ and proves that it is the unique solution to the backward Kolmogorov PDE \eqref{Backward:Kolmorgorov:PDE:semigroup}. We provide its proof for sake of completeness.




\begin{lemma}\label{existence:and:reg:sol:kolmogorov:smooth:setting}
Under the framework of Theorem \ref{thm:error:upper:bound:smooth:case}, the following conclusions hold:

\begin{itemize}
    \item [(i)] For any $1\leq i, j \leq d$, the derivatives $\partial^{i}_\mu \mathscr{U}(s,\mu)(\upsilon)$, $\partial_{\upsilon_j} \partial^{i}_\mu \mathscr{U}(s,\mu)(\upsilon)$ exist, are continuous on $[0,t]\times \mathcal{P}_2(\mathbb{R}^d) \times \mathbb{R}^d$, bounded and Lipschitz continuous with respect to $\mu$ and $\upsilon$ uniformly in $s\in [0,t]$.  
    \item[(ii)] For any $1\leq i, j \leq d$, the derivative $\partial^{(i, j)}_{\mu} \mathscr{U}(s, \mu)(\upsilon, \upsilon')$ is continuous on $[0,t]\times \mathcal{P}_2(\mathbb{R}^d) \times (\mathbb{R}^d)^2$ and bounded.
    \item[(iii)] The map $\mathscr{U}$ is the unique solution to the PDE \eqref{Backward:Kolmorgorov:PDE:semigroup} in the class of maps that belong to $\mathcal{C}^{1,2}([0,t]\times \mathcal{P}_2(\mathbb{R}^d))$ with bounded derivatives.  
\end{itemize}
\end{lemma}

\begin{proof}
Under the current regularity assumption on the coefficients $b_i$ and $\sigma_{i, j}$, it follows from Theorem 3.2 \cite{crisan:murray} that the derivatives $\partial_{x_i} \widetilde{X}_t^{s, x, \mu}$, $\partial^{2}_{x_i, x_j} \widetilde{X}_t^{s, x, \mu}$, $\partial^{i}_{\mu} \widetilde{X}_t^{s, x, \mu}(\upsilon)$, $\partial_\mu^{j} \partial_{x_i} \widetilde{X}_t^{s, x, \mu}(\upsilon)$ and $\partial^{(i, j)}_{\mu} \widetilde{X}_t^{s, x, \mu}(\upsilon)$, $1\leq i, j \leq d$, exist, are bounded and Lipschitz-continuous with respect to $x$, $\mu$ and $\upsilon$, in $L^{p}(\mathbb{P})$, for any $p\geq1$. 

Consequently, the derivatives $\partial^{i}_\mu \mathscr{U}(s,\mu)(\upsilon)$ and $\partial_{\upsilon_j} \partial^{i}_\mu \mathscr{U}(s,\mu)(\upsilon)$ exist and satisfy 
\begin{equation}\label{first:order:deriv:sol:kolmogorov:pde:smooth:setting}
\begin{aligned}
\partial^{i}_\mu \mathscr{U}(s,\mu)(\upsilon) & = \mathbb{E}[\partial^k_\mu \Phi(\mathcal{L}({X}^{s, \xi}_t))(\widetilde{X}_t^{s, \upsilon, \mu}) \partial_{x_i} (\widetilde{X}_t^{s, \upsilon, \mu})^k] \\
& \quad + \mathbb{E}[\partial^k_\mu \Phi(\mathcal{L}({X}^{s, \xi}_t))(\widetilde{X}_t^{s, \xi}) \partial^i_\mu (\widetilde{X}_t^{s, \xi})^k(\upsilon)]
\end{aligned}
\end{equation}

\noindent and
\begin{align*}
\partial_{\upsilon_j}\partial^{i}_\mu \mathscr{U}(s,\mu)(\upsilon) & = \mathbb{E}[\partial_{\upsilon_{\ell}} \partial^k_\mu \Phi(\mathcal{L}({X}^{s, \xi}_t))(\widetilde{X}_t^{s, \upsilon, \mu}) \partial_{x_j} (\widetilde{X}_t^{s, \upsilon, \mu})^\ell \partial_{x_i} (\widetilde{X}_t^{s, \upsilon, \mu})^k] \\
& \quad + \mathbb{E}[\partial^k_\mu \Phi(\mathcal{L}({X}^{s, \xi}_t))(\widetilde{X}_t^{s, \upsilon, \mu}) \partial^2_{x_i, x_j} (\widetilde{X}_t^{s, \upsilon, \mu})^k]\\
& \quad \quad + \mathbb{E}[\partial^k_\mu \Phi(\mathcal{L}({X}^{s, \xi}_t))(\widetilde{X}_t^{s, \xi}) \partial_{\upsilon_{j}}\partial^i_\mu (\widetilde{X}_t^{s, \xi})^k(\upsilon)],
\end{align*}

\noindent recalling that we use the summation over repeated indices convention and where $\partial^i_\mu (\widetilde{X}_t^{s, \xi})^k(\upsilon) = \partial^i_\mu (\widetilde{X}_t^{s, x, \mu})^k(\upsilon)_{|x=\xi}$, $\partial_{\upsilon_{j}}\partial^i_\mu (\widetilde{X}_t^{s, \xi})^k(\upsilon) = \partial_{\upsilon_{j}}\partial^i_\mu (\widetilde{X}_t^{s, x, \xi})^k(\upsilon)_{|x=\xi}$.

Hence, the aforementioned derivatives are continuous on $[0,t]\times \mathcal{P}_2(\mathbb{R}^d) \times \mathbb{R}^d$, bounded and Lipschitz continuous with respect to $\mu$ and $\upsilon$ uniformly in $s\in [0,t]$. We similarly obtain that $\partial^{(i,j)}_\mu \mathscr{U}(s,\mu)(\upsilon, \upsilon')$ exists and by taking the $\partial_\mu^{j}$ derivative on both sides of \eqref{first:order:deriv:sol:kolmogorov:pde:smooth:setting}, we obtain
\begin{equation}
    \begin{aligned}
        \partial_\mu^{(i,j)} \mathscr{U}(s,\mu)(\upsilon,\upsilon')&=\E[\hat\E[\partial_\mu^{(k,\ell)}\Phi(\mathcal{L}({X}^{s, \xi}_t))(\widetilde{X}_t^{s, \upsilon, \mu}, \widetilde{X}_t^{s, \upsilon', \mu})\partial_{x_j}(\widehat{{X}}_t^{s, \upsilon', \mu})^\ell\partial_{x_i}(\widetilde{X}_t^{s, \upsilon, \mu})^k]]\\
        &+\E[\partial_{\upsilon_{\ell}}\partial_\mu^k \Phi(\mathcal{L}({X}^{s, \xi}_t))(\widetilde{X}_t^{s, \upsilon, \mu}) \partial_\mu^j(\widetilde{X}_t^{s,\upsilon,\mu})^\ell(\upsilon')\partial_{x_i}(\widetilde{X}_t^{s, \upsilon, \mu})^k]\\
        &+\E[\partial_\mu^k \Phi(\mathcal{L}({X}^{s, \xi}_t))(\widetilde{X}_t^{s, \upsilon, \mu})\partial_\mu^j\partial_{x_i}(\widetilde{X}_t^{s, \upsilon, \mu})^k]\\
&+\E[\hat\E[\partial_\mu^{(k,\ell)}\Phi(\mathcal{L}({X}^{s, \xi}_t))(\widetilde{X}_t^{s, \xi},\widehat{{X}}_t^{s, \xi})\partial_\mu^j(\widehat{X}_t^{s, \xi})^\ell(\upsilon')\partial^i_\mu (\widetilde{X}_t^{s, \xi})^k(\upsilon)]]\\
        &+\E[\partial_{\upsilon_{\ell}}\partial_\mu^k \Phi(\mathcal{L}({X}^{s, \xi}_t))(\widetilde{X}_t^{s, \xi}) \partial_\mu^j(\widetilde{X}_t^{s,\xi})^\ell(\upsilon')\partial^i_\mu (\widetilde{X}_t^{s, \xi})^k(\upsilon)]\\
        &+\E[\partial_\mu^k \Phi(\mathcal{L}({X}^{s, \xi}_t))(\widetilde{X}_t^{s, \xi})\partial_\mu^{(i,j)} (\widetilde{X}_t^{s, \xi})^k(\upsilon,\upsilon')]\\
    \end{aligned}
\end{equation}

\noindent where $(\widehat{{X}}^{s,\upsilon,\mu}_t)_{s\leq t\leq T}$(resp. $(\widehat{{X}}^{s,\xi}_t)_{s\leq t\leq T}$) is a copy of the corresponding original processes $(\widetilde{X}^{s,\upsilon,\mu}_t)_{s\leq t\leq T}$(resp. $(\widetilde{X}^{s,\xi}_t)_{s\leq t\leq T}$) defined on a copy $(\hat\Omega,\hat\Fc,\hat\P)$ of the original probability space $(\Omega,\Fc,\P)$, and $\hat \E$ is the expectation
taken on $(\hat\Omega,\hat\Fc,\hat\P)$ with respect to $\hat\P$. From the previous identity, we similarly deduce that the derivative $\partial^{(i, j)}_{\mu} \mathscr{U}(s, \mu)(\upsilon, \upsilon')$ is continuous on $[0,t]\times \mathcal{P}_2(\mathbb{R}^d) \times (\mathbb{R}^d)^2$ and bounded.

It eventually follows from the Markov property satisfied by the unique solution to the McKean-Vlasov SDE \eqref{SDE:MCKEAN} that $s\mapsto \mathscr{U}(s, \mu)$ is continuously differentiable and that $\partial_t \mathscr{U}(s, \mu) = - \mathscr{L}_s \mathscr{U}(s, \mu)$ for any $(s,\mu) \in [0,t) \times \mathcal{P}_2(\mathbb{R}^d)$. 

In order to prove uniqueness, we remark that It\^o's rule guarantees that $\Phi(\mathcal{L}({X}^{s, \xi}_t)) = U(t, \mathcal{L}({X}^{s, \xi}_t)) = U(s, \mu)$ for any solution $U \in \mathcal{C}^{1,2}([0,t]\times \mathcal{P}_2(\mathbb{R}^d))$ to the PDE \eqref{Backward:Kolmorgorov:PDE:semigroup} with bounded derivatives. Hence, $\mathscr{U}(s, \mu) = U(s, \mu)$ for any $(s,\mu)\in [0,t]\times \mathcal{P}_2(\mathbb{R}^d)$ which proves uniqueness.

\end{proof}

\subsection{Proof of Theorem \ref{Thm:Conv:EulerMaru:semigroup}}\label{proof:Thm:Conv:EulerMaru:semigroup}
The main idea is to write the following key decomposition 
 \begin{equation}\label{key:decomposition:theorem:weak:strong:error:semigroup}
 \Phi(\mu^{N,n}_t) - \Phi(\mathcal{L}({X}^{\xi}_t)) = \mathscr{U}(0, \mu^{N}_0) - \mathscr{U}(0,\mu) + \mathscr{U}(t, \mu^{N,n}_t) - \mathscr{U}(0, \mu^{N}_0).
 \end{equation}

As previously mentioned, instead of working directly with $\mathscr{U}$ to deal with the two terms in the above decomposition, we will work with the approximation sequence $(\mathscr{U}^{(m)})_{m\geq1}$ studied earlier.\\

\noindent \emph{Proof of \eqref{esti:sem:chaos}}\\

 \emph{Step 1:}
The first term of the decomposition \eqref{key:decomposition:theorem:weak:strong:error:semigroup} corresponds to the error between the empirical measure at time $0$ and the initial law $\mu$ when both acts on the unique solution to the backward Kolmogorov PDE \eqref{Backward:Kolmorgorov:PDE:semigroup}. We approximate it by $\mathscr{U}^{(m)}(0, \mu^{N}_0) - \mathscr{U}^{(m)}(0,\mu)$. Assuming that the random variables $(\xi^{i})_{1\leq i \leq N}$ are i.i.d., it follows from Theorem 2.11 \cite{chassagneux:szpruch:tse} (see equation (2.21) therein) that
\begin{equation}\label{init:error:weak:prop:chaos}
\begin{aligned}
 \E[\mathscr{U}^{(m)}(0, \mu^{N}_0)  ] & - \mathscr{U}^{(m)}(0,\mu) \\
  & = \frac{1}{N} \mathbb{E}\Big[\int_{[0,1]^2}  \int_{\mathbb{R}^d} \lambda  \frac{\delta^2 \mathscr{U}^{(m)}}{\delta m^2}(0, m_{\lambda, \lambda_1}^{N})(\tilde{\xi}_1, y) (\delta_{\tilde{\xi}_1} - \delta_{\xi_1})(\d y)  \, \d \lambda \d \lambda_1 \Big]
\end{aligned}
\end{equation}
\noindent where $\tilde{\xi}^1$ is a random variable with law $\mu$ that is independent of $(\xi^{i})_{1\leq i \leq N}$ and
\begin{equation}
\begin{aligned}
m^{N}_{\lambda, \lambda_1} & : = \mu + \lambda (\mu_0^{N}-\mu)+  \lambda \lambda_1 N^{-1} (\delta_{\tilde{\xi}^1} - \delta_{\xi^1}).
\end{aligned}
\end{equation}

 Now, Lemma 2.5 in \cite{chassagneux:szpruch:tse} together with Lemma \ref{lemma:estimate:deriv:um} guarantees that
$$
\Big|  \frac{\delta^2 \mathscr{U}^{(m)}}{\delta m^2}(0, m_{\lambda, \lambda_1}^{N})(\upsilon_1,  \upsilon_2)\Big| \leq C \frac{|\upsilon_1| |\upsilon_2|}{t^{1-\frac{\alpha}{2}}} (1+|\upsilon_1|+|\upsilon_2| + M_2(m^N_{\lambda, \lambda_1})).
$$

\noindent Hence, using the fact that $M_4(\mu)^4=\mathbb{E}[|\xi^{1}|^4]<\infty$, we get 
\begin{equation}\label{weak:error:bound:initial:condition}
\Big|\E[\mathscr{U}^{(m)}(0, \mu^{N}_0) ]  - \mathscr{U}^{(m)}(0, \mu) \Big| \leq \frac{C}{t^{1-\frac{\alpha}{2}}N}.
\end{equation}

\noindent \emph{Step 2: }
The second term of the decomposition \eqref{key:decomposition:theorem:weak:strong:error:semigroup} is approximated by $\mathscr{U}^{(m)}(t, \mu^{N,n}_t) - \mathscr{U}^{(m)}(0, \mu^{N}_0)$. In order to proceed, we apply It\^o's rule to $(\mathscr{U}^{(m)}(s,\mu^{N, n}_s))_{s \in [0,t]}$. Hence,
\begin{equation*}
\begin{aligned}
&\mathscr{U}^{(m)}(s,\mu^{N,n}_s)=\mathscr{U}^{(m)}(0,\mu^N_0)+\int_0^s(\partial_t+\mathscr{L}_{s_1})\mathscr{U}^{(m)}(s_1,\mu^{N,n}_{s_1}) \, \d s_1\\
&+\frac{1}{N}\sum_{i=1}^N\int_0^s (b(k_n(s_1),X^{i,n}_{k_n(s_1)},\mu^{N,n}_{k_n(s_1)})-b(s_1,X^{i,n}_{s_1},\mu^{N,n}_{s_1}))\cdot \partial_\mu \mathscr{U}^{(m)}(s_1,\mu^{N,n}_{s_1})(X^{i,n}_{s_1}) \, \d s_1\\
&+\frac{1}{N}\sum_{i=1}^N\int_0^s \text{trace}\big([a(k_n(s_1),X^{i,n}_{k_n(s_1)},\mu^{N,n}_{k_n(s_1)})-a(s,X^{i,n}_{s_1},\mu^{N,n}_{s_1})]\partial_\upsilon \partial_\mu \mathscr{U}^{(m)}(s_1, \mu^{N,n}_{s_1})(X^{i,n}_{s_1})\big) \, \d s_1\\
&+\frac{1}{N}\sum_{i=1}^N\int_0^s \partial_\mu \mathscr{U}^{(m)}(s_1,\mu^{N,n}_{s_1})(X^{i,n}_{s_1})\sigma(k_n(s_1),X^{i,n}_{s_1},\mu^{N,n}_{k_n(s_1)})\, \d W^{(i)}_{s_1}\\
&+\frac{1}{2N^2}\sum_{i=1}^N\int_0^s \text{trace}(a(k_n(s_1),X^{i,n}_{k_n(s_1)},\mu^{N,n}_{k_n(s_1)})\partial^2_\mu \mathscr{U}^{(m)}(s_1,\mu^{N,n}_{s_1})(X^{i,n}_{s_1},X^{i,n}_{s_1})) \, \d s_1.
\end{aligned}
\end{equation*}





Note that \eqref{first:deriv:mes:U} guarantees that the stochastic integral appearing in the right-hand side of the above identity is a martingale. Hence, taking expectation in the previous identity and then letting $s=t$, we get
\begin{equation}\label{key:decomposition:expected:Um}
    \begin{aligned}
    & \E[   \mathscr{U}^{(m)}(t,\mu^{N,n}_t)  - \mathscr{U}^{(m)}(0,\mu^N_0)] =\E\Big[\int_0^t (\partial_t+\mathscr{L}_{s_1})\mathscr{U}^{(m)}(s_1,\mu^{N,n}_{s_1}) \, \d s_1\Big]\\
    &+\frac{1}{N}\sum_{i=1}^N\E\Big[\int_0^t (b(k_n(s_1),X^{i,n}_{k_n(s_1)},\mu^{N,n}_{k_n(s_1)})-b(s_1,X^{i,n}_{s_1},\mu^{N,n}_{s_1}))\cdot \partial_\mu \mathscr{U}^{(m)}(s_1,\mu^{N,n}_{s_1})(X^{i , n}_{s_1}) \, \d s_1\Big]\\
&+\frac{1}{N}\sum_{i=1}^N\E\Big[\int_0^t \text{trace}\Big([a(k_n(s_1),X^{i,n}_{k_n(s_1)},\mu^{N,n}_{k_n(s_1)})-a(s,X^{i,n}_{s_1},\mu^{N,n}_{s_1})]\partial_\upsilon \partial_\mu \mathscr{U}^{(m)}(s_1, \mu^{N,n}_{s_1})(X^{i,n}_{s_1})\big) \, \d s_1\Big]\\
&+\frac{1}{2N^2}\sum_{i=1}^N\E\Big[\int_0^t \text{trace}(a(k_n(s_1),X^{i,n}_{k_n(s_1)},\mu^{N,n}_{k_n(s_1)})\partial^2_\mu \mathscr{U}^{(m)}(s_1,\mu^{N,n}_{s_1})(X^{i,n}_{s_1},X^{i,n}_{s_1})) \, \d s_1\Big]\\
& := {\rm I}^m_1 + {\rm I}^m_2 + {\rm I}^m_3 + {\rm I}^m_4.
    \end{aligned}
\end{equation}

\noindent \emph{Step 3:}

We here establish appropriate upper-bounds for ${\rm I}^m_2$, ${\rm I}^m_3$ and ${\rm I}^m_4$ and then prove that ${\rm I}^m_1$ vanishes as $m\uparrow \infty$. 
First, the H\"older regularity of the coefficients $b_i$ together with \eqref{first:deriv:mes:U} gives
\begin{align*}
|{\rm I}^m_2| &\leq K \int_{0}^{t} \frac{1}{(t-s_1)^{\frac{1-\alpha}{2}}} \Big((s_1-k_n(s_1))^{\frac{\eta}{2}} +  \Big(\frac{1}{N}\sum_{i=1}^N \mathbb{E}[|X^{i,n}_{s_1}-X^{i,n}_{k_n(s_1)}|^2]\Big)^{\frac{\eta}{2}}\Big) \, \d s_1\\
& \leq C t^{\frac{1+\alpha}{2}} h^{\frac{\eta}{2}}.
\end{align*}

We deal with ${\rm I}^m_3$ and ${\rm I}^m_4$ in a similar way. Namely, the H\"older regularity and boundedness of $a_{i,j}$ together with \eqref{cross:deriv:mes:upsilon:U} yields
$$
|{\rm I}^m_3| \leq C t^{\frac{\alpha}{2}} h^{\frac{\eta}{2}} \quad \mbox{ and } \quad |{\rm I}^m_4| \leq C \frac{t^{\frac{\alpha}{2}}}{N} .
$$

Gathering the three previous bounds and coming back to \eqref{key:decomposition:expected:Um}, we obtain 
\begin{equation}\label{bound:Um:before:limit}
\Big| \E[   \mathscr{U}^{(m)}(t,\mu^{N,n}_t)  - \mathscr{U}^{(m)}(0,\mu^N_0)] \Big| \leq \Big| {\rm I}^{m}_1\Big| + K t^{\frac{\alpha}{2}}(h^{\frac{\eta}{2}}+N^{-1}).
\end{equation}

We now pass to the limit as $m \uparrow \infty$ in the previous inequality. In order to deal with ${\rm I}^m_1$, we use the following decomposition
\begin{align*}
{\rm I}^m_1 & = \E\Bigg[\int_0^t (\partial_t+\mathcal{L}_{s_1}) \mathscr{U}^{(m)}(s_1,\mu^{N,n}_{s_1}) \textbf{1}_{|\mathbf{X}_{s_1}^{N,n}|\leq R}\, \d s_1\Bigg] + \E\Bigg[\int_0^t (\partial_t+\mathcal{L}_{s_1})\mathscr{U}^{(m)}(s_1,\mu^{N,n}_{s_1}) \textbf{1}_{|\mathbf{X}_{s_1}^{N,n}| >  R}\, \d s_1\Bigg] \\
& =: {\rm I}^{m, R}_{1, 1} + {\rm I}^{m, R}_{1,2}, 
\end{align*}

\noindent where $R>0$ is a fixed parameter.

According to \cite{CHAUDRUDERAYNAL20221, CHAUDRUDERAYNAL20211}, there exists a subsequence of $(s, \mu) \mapsto \partial_t\mathscr{U}^{(m)}(s, \mu)$, $(s, \mu, \upsilon)\mapsto \partial_\mu \mathscr{U}^{(m)}(s, \mu)(\upsilon)$, \, $\partial_\upsilon \partial_\mu \mathscr{U}^{(m)}(s, \mu)(\upsilon)$, $m \geq1$, still denoted by $\partial_t \mathscr{U}^{(m)}$, $\partial_\mu \mathscr{U}^{(m)}$ and $\partial_\upsilon \partial_\mu \mathscr{U}^{(m)}$, that converges to $\partial_t \mathscr{U}(s,\mu)$,  $\partial_\mu \mathscr{U}(s,\mu)(\upsilon)$ and $\partial_\upsilon \partial_\mu \mathscr{U}(s,\mu)(\upsilon)$ uniformly on compact sets of $  [0,T) \times \mathcal{P}_2(\mathbb{R}^d)$ and $  [0,T) \times \mathcal{P}_2(\mathbb{R}^d) \times \mathbb{R}^d$ respectively. 

Hence, for any $R>0$, $(\partial_t + \mathscr{L}_{s_1})  \mathscr{U}^{(m)}(s_1,\mu^{N,n}_{s_1}) \textbf{1}_{|\mathbf{X}^{N, n}_{s_1}|\leq R }$ converges $a.s.$ to $(\partial_t + \mathscr{L}_{s_1})  \mathscr{U}(s_1, \mu^{N,n}_{s_1}) \textbf{1}_{|\mathbf{X}^{N, n}_{s_1}|\leq R }$ which in turn using the boundedness of the coefficients $b_i$ and $a_{i,j}$, the estimates \eqref{first:deriv:mes:U}, \eqref{cross:deriv:mes:upsilon:U} and the dominated convergence theorem yield
\begin{align}
\lim_{m\uparrow \infty} {\rm I}^{m, R}_{1,1} & = \mathbb{E}\Big[\int_0^t  (\partial_t + \mathscr{L}_{s_1})  \mathscr{U}(s_1, \mu^{N,n}_{s_1}) \textbf{1}_{|\mathbf{X}^{N, n}_{s_1}|\leq R } \, \d s_1 \Big],\label{esti:lim:approx:part_temp}
\end{align}

\noindent and, again by dominated convergence, 
\begin{align*}
  \lim_{R\uparrow \infty} \mathbb{E}\Big[\int_0^t  (\partial_t + \mathscr{L}_{s_1})  \mathscr{U}(s_1, \mu^{N,n}_{s_1}) \textbf{1}_{|\mathbf{X}^{N, n}_{s_1}|\leq R } \, \d s_1 \Big] = \mathbb{E}\Big[\int_0^t  (\partial_t + \mathscr{L}_{s_1})  \mathscr{U}(s_1,\mu^{N,n}_{s_1}) \, \d s_1 \Big] = 0, 
\end{align*}
 \noindent since $[0,t]\times \mathcal{P}_2(\mathbb{R}^d) \ni (s, \mu) \mapsto  \mathscr{U}(s,\mu) $ is the unique solution to the PDE \eqref{Backward:Kolmorgorov:PDE:semigroup}. Combining the two previous limits, we get $\lim_{R\uparrow \infty}\lim_{m\uparrow \infty} {\rm I}_{1,1}^{m,R} = 0$.
 
 To deal with ${\rm I}^{m,R}_{1,2}$, we use the estimates \eqref{first:deriv:mes:U}, \eqref{cross:deriv:mes:upsilon:U} as well as the boundedness of the coefficients to get 
 \begin{align*}
 \lim\sup_{R\uparrow \infty} \sup_{m\geq1}\Big| {\rm I}_{1, 2}^{m, R} \Big| & \leq  C \lim \sup_{R\uparrow \infty}\mathbb{E}\Big[\int_0^t \frac{1+M_2(\mu_{s_1}^{N,n})}{(t-s_1)^{1-\frac{\alpha}{2}}}\textbf{1}_{|\mathbf{X}^{N, n}_{s_1}|\geq R } \, \d s_1\Big] \\
 & \leq K \lim_{R\uparrow \infty}\mathbb{P}(\sup_{0\leq s\leq t}|\mathbf{X}_s^{N,n}|\geq R)^{\frac12} = 0.
 \end{align*}

Hence, combining the previous limits, namely, passing to the limit in $m$ along the aforementioned subsequence and then in $R$ gives
$$
\lim_{m} {\rm I}_1^{m} = 0.
$$

 A similar argument gives 
 $$
 \lim_{m}\E\Big[ \mathscr{U}^{(m)}(t, \mu^{N, n}_t) -  \mathscr{U}^{(m)}(0, \mu_0^N) \Big] = \mathbb{E}[ \mathscr{U}(t, \mu^{N,n}_t) -  \mathscr{U}(0, \mu_0^N)] = \mathbb{E}[\Phi(\mu^{N,n}_t)- \mathscr{U}(0, \mu_0^{N})].
 $$

Hence, passing to the limit in $m$ in the inequality \eqref{bound:Um:before:limit}, we obtain 
$$
|\mathbb{E}[\Phi(\mu^{N,n}_t) -  \mathscr{U}(0, \mu_0^{N})]| \leq K t^{\frac{\alpha}{2}}(h^{\frac{\eta}{2}}+N^{-1}).
$$

Similarly, passing to the limit as $m \uparrow \infty$ in \eqref{weak:error:bound:initial:condition} gives
$$
\Big|\E[\mathscr{U}(0, \mu^{N}_0) ]  - \mathscr{U}(0, \mu) \Big| = \lim_m \Big|\E[\mathscr{U}^{(m)}(0, \mu^{N}_0) ]  - \mathscr{U}^{(m)}(0, \mu) \Big| \leq \frac{K}{t^{1-\frac{\alpha}{2}}N}.
$$

The estimate \eqref{esti:sem:chaos} is eventually obtained by coming back to \eqref{key:decomposition:theorem:weak:strong:error:semigroup} and by gathering the two previous estimates.\\

\noindent \emph{Proof of \eqref{strong:conv:chaos}:}\\

We come back to the key decomposition \eqref{key:decomposition:theorem:weak:strong:error:semigroup}.\\

\emph{Step 1:}
We first provide a bound for $\mathbb{E}[|\mathscr{U}(0, \mu^{N}_0) - \mathscr{U}(0,\mu)|]$. It follows from the very definition of the flat derivative and the identity $\partial_\mu \mathscr{U}(0, \mu)(\upsilon) = \partial_{\upsilon} \frac{\delta \mathscr{U}}{\delta m}(0, \mu)(\upsilon)$ that for any probability measures $\mu, \nu \in \mathcal{P}_2(\mathbb{R}^d)$ and any coupling $\pi$ between $\mu$ and $\nu$ that
\begin{align*}
\mathscr{U}(0, \mu) - \mathscr{U}(0, \nu) & = \int_0^1 \int_{(\mathbb{R}^d)^2} \left\{\frac{\delta \mathscr{U}}{\delta m}(0, \lambda \mu + (1-\lambda) \nu)(y) -\frac{\delta \mathscr{U}}{\delta m}(0, \lambda \mu + (1-\lambda) \nu)(x) \right\} \, \pi(\d x, \d y) \, \d \lambda\\
& = \int_{[0,1]^2} \int_{(\mathbb{R}^d)^2} \partial_\mu \mathscr{U}(0, \lambda \mu + (1-\lambda) \nu)(\lambda' y+(1-\lambda')x).(x-y) \, \pi(\d x, \d y) \, \d \lambda \, \d \lambda'.
\end{align*}

Hence, from \eqref{first:deriv:mes:U} and the previous identity, we deduce
$$
|\mathscr{U}(0, \mu) - \mathscr{U}(0, \nu)| \leq C t^{\frac{\alpha-1}{2}} (1+M_2(\mu)+M_2(\nu))\mathcal{W}_2(\mu, \nu)
$$

\noindent which in turn implies
\begin{equation}\label{strong:error:initial:condition}
\mathbb{E}[|\mathscr{U}(0, \mu^{N}_0) - \mathscr{U}(0,\mu)|] \leq C t^{\frac{\alpha-1}{2}} \mathbb{E}[\mathcal{W}_2(\mu_0^N, \mu)^2]^{\frac12}.
\end{equation}

\noindent \emph{Step 2:}
Recalling It\^o's rule applied to 
 $(\mathscr{U}^{(m)}(s,\mu^{N, n}_s))_{s \in [0,t]}$, we obtain
\begin{equation}\label{key:decomposition:expected:Um:WithAbsVal}
\begin{aligned}
&\lvert \mathscr{U}^{(m)}(t,\mu^{N,n}_t)- \mathscr{U}^{(m)}(0,\mu^N_0)\rvert\leq \Big| \int_0^t(\partial_t+\mathscr{L}_{s_1}) \mathscr{U}^{(m)}(s_1,\mu^{N,n}_{s_1}) \, \d s_1 \Big|\\\
&+\frac{1}{N}\sum_{i=1}^N\int_0^t \lvert b(k_n(s_1),X^{i,n}_{k_n(s_1)},\mu^{N,n}_{k_n(s_1)})-b(s_1,X^{i,n}_{s_1},\mu^{N,n}_{s_1})\rvert\lvert \partial_\mu \mathscr{U}^{(m)}(s_1,\mu^{N,n}_{s_1})(X^{i,n}_{s_1})\rvert \, \d s_1\\
&+\frac{1}{N}\sum_{i=1}^N\int_0^t \lvert a(k_n(s_1),X^{i,n}_{k_n(s_1)},\mu^{N,n}_{k_n(s_1)})-a(s_1,X^{i,n}_{s_1},\mu^{N,n}_{s_1})\rvert\lvert\partial_\upsilon \partial_\mu \mathscr{U}^{(m)}(s_1,\mu^{N,n}_{s_1})(X^{i,n}_{s_1})\rvert \, \d s_1\\
&+\Bigg|\int_0^t \frac{1}{N}\sum_{i=1}^N \partial_\mu \mathscr{U}^{(m)}(s_1,\mu^{N,n}_{s_1})(X^{i,n}_{s_1})\sigma(k_n(s_1),X^{i,n}_{s_1},\mu^{N,n}_{k_n(s_1)})\, \d W^{(i)}_{s_1}\Bigg|\\
&+\frac{1}{2N^2}\sum_{i=1}^N\int_0^t \lvert a(k_n(s_1),X^{i,n}_{k_n(s_1)},\mu^{N,n}_{k_n(s_1)})\rvert\lvert\partial^2_\mu \mathscr{U}^{(m)}(s_1,\mu^{N,n}_{s_1})(X^{i,n}_{s_1},X^{i,n}_{s_1}))\rvert \, \d s_1\\
:&={\rm J}^m_1+ {\rm J}^m_2 + {\rm J}^m_3+ {\rm J}^m_4+ {\rm J}^m_5.
\end{aligned}
\end{equation}

From similar arguments as those used in the last step of the proof of \eqref{esti:sem:chaos}, we get
$$
\lim_{m} \mathbb{E}[|{\rm J}^m_1|] = 0
$$

\noindent and 
$$
\lim_{m} \mathbb{E}[\lvert \mathscr{U}^{(m)}(t,\mu^{N,n}_t)- \mathscr{U}^{(m)}(0,\mu^N_0)\rvert] = \mathbb{E}[|\Phi(\mu^{N,n}_t) - \mathscr{U}(0, \mu_0^N)|].
$$

As in the proof of \eqref{esti:sem:chaos}, using the H\"older regularity of $a_{i, j}$ and $b_i$ as well as \eqref{first:deriv:mes:U} and \eqref{cross:deriv:mes:upsilon:U}, we get
$$
\E [|{\rm J}^m_2|]+\E [|{\rm J}^m_3|]\leq C t^{\frac{\alpha}{2}}h^{\frac{\eta}{2}}.
$$

Using the Cauchy-Schwarz inequality and \eqref{first:deriv:mes:U}, we obtain
\begin{align*}
\mathbb{E}[|{\rm J}_4^{m}|]
& \leq  \frac{C}{\sqrt{N}} \mathbb{E}\Big[\Big(\int_0^t \frac{1+M_2(\mu_{s_1}^{N,n})^2 +|X_{s_1}^{i,n}|^2}{(t-s_1)^{1-\alpha}} \, \d s_1 \Big)^{\frac{1}{2}}\Big] \leq C \frac{t^{\frac{\alpha}{2}}}{\sqrt{N}}.
\end{align*}

The boundedness of $a$ together with \eqref{cross:deriv:mes:upsilon:U} yields
$$
\E[|{\rm J}^m_5|] \leq C \frac{t^{\frac{\alpha}{2}}}{N}. 
$$

Gathering the previous estimates and limits, we conclude
$$
\mathbb{E}[|\Phi(\mu^{N,n}_t) - \mathscr{U}(0, \mu_0^N)|] \leq C t^{\frac{\alpha}{2}} \Big(  h^{\frac{\eta}{2}} + \frac{1}{\sqrt{N}} \Big)
$$

\noindent which together with \eqref{strong:error:initial:condition} allows to complete the proof of \eqref{strong:conv:chaos}.

\subsection{Proof of Theorem \ref{thm:error:upper:bound:smooth:case}}\label{proof:thm:error:upper:bound:smooth:case}

As in the proof of Theorem \ref{Thm:Conv:EulerMaru:semigroup}, we rely on the key identity \eqref{key:decomposition:theorem:weak:strong:error:semigroup} and study separately the two terms.

\noindent \emph{Step 1:}

From the identity \eqref{init:error:weak:prop:chaos} which is valid for $\mathscr{U}$ instead of $\mathscr{U}^{(m)}$, the uniform boundedness of $\partial^{(i , j)}_\mu \mathscr{U}(s,\mu)$ and Lemma 2.5 in \cite{chassagneux:szpruch:tse}, we get
\begin{equation}\label{weak:error:bound:initial:condition:smooth:setting}
\Big|\E[\mathscr{U}(0, \mu^{N}_0) ]  - \mathscr{U}(0, \mu) \Big| \leq \frac{C}{N}.
\end{equation}

It thus remains to prove that
\begin{equation}\label{weak:error:bound:diff:PDE:smooth:setting}
\sup_{0\leq t \leq T}|\mathbb{E}[\Phi(\mu_t^{N,n}) - \mathscr{U}(0,\mu_0^N)]| \leq C (N^{-1}+h).
\end{equation}

Now, from It\^o's rule applied to $(\mathscr{U}(s,\mu_s^{N,n}))_{0\leq s\leq t}$, it holds
 \begin{equation}\label{itos:rule:test:function:smooth:setting}
\begin{aligned}
& \mathbb{E}[ \Phi(\mu^{N, n}_{t})  - \mathscr{U}(0,\mu_0^N) ] \\
& = 
\int_0^t  \E\Big[ [b_i(k_n(s_1), X^{1, n}_{k_n(s_1)}, \mu^{N, n}_{k_n(s_1)}) - b_i(s_1, X^{1, n}_{s_1}, \mu^{N, n}_{s_1})]  \partial^i_\mu \mathscr{U}(s_1, \mu^{N, n}_{s_1})(X^{1, n}_{s_1}) \Big] \, \d s_1  \\
&  + \int_0^t \frac12 \E\Big[ [ a_{i, j}(k_n(s_1), X^{1, n}_{k_n(s_1)}, \mu^{N, n}_{k_n(s_1)}) -  a_{i, j}(s_1, X^{1, n}_{s_1}, \mu^{N, n}_{s_1})]  \partial_{\upsilon_j} \partial^{i}_\mu  \mathscr{U}(s_1, \mu^{N, n}_{s_1})(X^{1, n}_{s_1})  \Big] \, \, \d s_1\\
& + \int_0^t \frac{1}{2N} \E\left[ a_{i,j}(k_n(s_1), X^{1, n}_{k_n(s_1)}, \mu^{N, n}_{k_n(s_1)})  \partial^{(i,j)}_\mu \mathscr{U}(s_1, \mu^{N, n}_{s_1})(X^{1, n}_{s_1}, X^{1, n}_{s_1}) \right] \, \d s_1. 
\end{aligned}
\end{equation}

We now decompose the integrand in the $\d s_1$-integral of the first term in the right-hand side of \eqref{itos:rule:test:function:smooth:setting} as follows
\begin{align*}
\E\Big[ & [b_i(k_n(s_1), X^{1, n}_{k_n(s_1)}, \mu^{N, n}_{k_n(s_1)})  - b_i(s_1, X^{1, n}_{s_1}, \mu^{N, n}_{s_1})]  \partial^i_\mu \mathscr{U}(s_1, \mu^{N, n}_{s_1})(X^{1, n}_{s_1}) \Big] \\
& = \E\Big[ [b_i(k_n(s_1), X^{1, n}_{k_n(s_1)}, \mu^{N, n}_{k_n(s_1)}) - b_i(s_1, X^{1, n}_{s_1}, \mu^{N, n}_{s_1})] \\
& \quad \times [ \partial^i_\mu \mathscr{U}(s_1, \mu^{N, n}_{s_1})(X^{1, n}_{s_1}) - \partial^i_\mu \mathscr{U}(s_1, \mu^{N, n}_{k_n(s_1)})(X^{1, n}_{k_n(s_1)})] \Big]\\
& \quad + \E\Big[ [b_i(k_n(s_1), X^{1, n}_{k_n(s_1)}, \mu^{N, n}_{k_n(s_1)}) - b_i(s_1, X^{1, n}_{s_1}, \mu^{N, n}_{s_1})]   \partial^i_\mu \mathscr{U}(s_1, \mu^{N, n}_{k_n(s_1)})(X^{1, n}_{k_n(s_1)}) \Big]\\
& =: {\rm E}^{i, n, N}_1(s_1) + {\rm E}^{i, n, N}_2(s_1).
\end{align*}

Using the Lipschitz regularity of the mappings $(s, x) \mapsto b_i(s, x, \mu)$ (uniformly in $\mu$) and $(\mu, \upsilon) \mapsto \partial_\mu^i \mathscr{U}(s,\mu)(\upsilon)$, recalling that $\partial_\mu b_i(s, x, \mu)(\upsilon)$, $\partial_\upsilon \partial_\mu^i \mathscr{U}(s, \mu)(\upsilon)$ and $\partial_\mu \partial_\mu^i \mathscr{U}(s, \mu)(\upsilon)$ are uniformly bounded, we get 
\begin{align*}
|{\rm E}^{i, n, N}_1(s_1)| & \leq K \mathbb{E}\Big[ [ (s_1-k_n(s_1)) + |X^{1, n}_{s_1} - X^{1, n}_{k_n(s_1)}| + \mathcal{W}_1(\mu^{N, n}_{s_1}, \mu^{N,n}_{k_n(s_1)}) ] \\
& \quad \times [|X^{1, n}_{s_1} - X^{1, n}_{k_n(s_1)}| +  \mathcal{W}_1(\mu^{N, n}_{s_1}, \mu^{N,n}_{k_n(s_1)})] \Big]. 
\end{align*}

It then follows from the standard inequality $\mathcal{W}_1(\mu^{N, n}_{s_1}, \mu^{N,n}_{k_n(s_1)})\leq \frac{1}{N}\sum_{l=1}^{N} |X^{l, n}_{s_1}-X^{l,n}_{k_n(s_1)}|$ and the standard estimate $\sup_{0\leq s_1 \leq T}\max_{1\leq l \leq N} \mathbb{E}[|X^{l, n}_{s_1}-X^{l, n}_{k_n(s_1)}|^2]\leq K h $ that
$$
\sup_{0\leq s_1\leq T}\max_{1\leq i \leq d}|{\rm E}^{i, n, N}_1(s_1)| \leq K h.
$$

We deal with ${\rm E}^{i, n, N}_2(s_1)$ by using It\^o's rule together with the uniform boundedness of the derivatives $\partial_s b_i(s, x, \mu)$, $\partial_x b_i(s, x, \mu)$, $\partial^2_x b_i(s, x, \mu)$,  $\partial_\mu b_i(s, x, \mu)$, $\partial_\upsilon \partial_\mu b_i(s, x, \mu)$, $\partial^2_\mu b_i(s, x, \mu)$ and $ \partial^i_\mu \mathscr{U}(s_1, \mu)$ to deduce that
$$
\sup_{0\leq s_1 \leq T} \max_{1\leq i \leq d}|{\rm E}^{i, n, N}_2(s_1)| \leq K h.
$$

Combining the two previous estimates, we conclude that
\begin{equation}\label{first:term:estimate:weak:error}
\begin{aligned}
\Big| \int_0^t  \E\Big[ [b_i(k_n(s_1), X^{1, n}_{k_n(s_1)}, \mu^{N, n}_{k_n(s_1)}) & - b_i(s_1, X^{1, n}_{s_1}, \mu^{N, n}_{s_1})]  \partial^i_\mu \mathscr{U}(s_1, \mu^{N, n}_{s_1})(X^{1, n}_{s_1}) \Big] \, \d s_1\Big|\\ & \leq K t h.
\end{aligned}
\end{equation}

In order to deal with the second term in \eqref{itos:rule:test:function:smooth:setting}, we remark that the derivatives $\partial_s a_{i, j}(s, x, \mu)$, $\partial_x a_{i, j}(s, x, \mu)$, $\partial^2_x a_{i, j}(s, x, \mu)$,  $\partial_\mu a_{i, j}(s, x, \mu)(\upsilon)$, $\partial_\upsilon \partial_\mu a_{i, j}(s, x, \mu)(\upsilon)$, $\partial^2_\mu a_{i, j}(s, x, \mu)(\upsilon)$ are of affine growth under our current assumption and that the random variable $\partial_{\upsilon_{j}}\partial^i_\mu \mathscr{U}(s_1, \mu^{N, n}_{s_1})(X^{1, n}_{s_1})$ is bounded. Then, by a similar decomposition as for the second term and from similar computations, we deduce that
\begin{equation}\label{second:term:estimate:weak:error}
\begin{aligned}
\Big| \int_0^t  \E\Big[ [a_{i, j}(k_n(s_1), X^{1, n}_{k_n(s_1)}, \mu^{N, n}_{k_n(s_1)}) & - a_{i, j}(s_1, X^{1, n}_{s_1}, \mu^{N, n}_{s_1})]  \partial_{\upsilon_{j}}\partial^i_\mu \mathscr{U}(s_1, \mu^{N, n}_{s_1})(X^{1, n}_{s_1}) \Big] \, \d s_1\Big|\\ & \leq K t h.
\end{aligned}
\end{equation}

We now come back to \eqref{itos:rule:test:function:smooth:setting}, use \eqref{first:term:estimate:weak:error}, \eqref{second:term:estimate:weak:error} and the boundedness of $\partial^{(i,j)}_\mu \mathscr{U}(s, \mu)(\upsilon, \upsilon)$. Hence, 
$$
\sup_{0\leq t \leq T} |\mathbb{E}[ \Phi(\mu^{N, n}_{t})  - \mathscr{U}(0,\mu_0^N) ] | \leq K T (h+N^{-1}).
$$

The previous bound combined with \eqref{weak:error:bound:initial:condition:smooth:setting} allows to conclude the proof.

\bibliography{bibliography.bib}

\begin{thebibliography}{10}

\bibitem{Oben2019}
O.~Bencheikh and B.~Jourdain.
\newblock Bias behaviour and antithetic sampling in mean-field particle
  approximations of sdes nonlinear in the sense of mckean.
\newblock {\em ESAIM: ProcS}, 65:219 -- 235, 2019.

\bibitem{ForBackwardSDENorCons}
P.~Briand, P.~Cardaliaguet, P.-E.~Chaudru de~Raynal, and Y.~Hu.
\newblock Forward and backward stochastic differential equations with normal
  constraint in law.
\newblock {\em Stochastic Processes and their Applications},
  130(12):7021--7097, 2020.

\bibitem{buckdahn2017}
R.~Buckdahn, J.~Li, S.~Peng, and C.~Rainer.
\newblock Mean-field stochastic differential equations and associated pdes.
\newblock {\em Ann. Probab.}, 45(2):824--878, 2017.

\bibitem{cardaliaguet}
P.~Cardaliaguet.
\newblock Notes on mean field games.
\newblock \url{https://www.ceremade.dauphine.fr/ cardaliaguet/MFG20130420.pdf},
  2013.

\bibitem{cardaliaguet2019master}
Pierre Cardaliaguet, Fran{\c{c}}ois Delarue, Jean-Michel Lasry, and
  Pierre-Louis Lions.
\newblock {\em The master equation and the convergence problem in mean field
  games}.
\newblock Princeton University Press, 2019.

\bibitem{cardel19}
R.~Carmona and F.~Delarue.
\newblock {\em Probabilistic Theory of Mean Field Games: vol. I, Mean Field
  FBSDEs, Control, and Games, Mean Field game with common noise and Master
  equations}.
\newblock Springer, 2018.

\bibitem{chassagneux:szpruch:tse}
J.-F. Chassagneux, L.~Szpruch, and A.~Tse.
\newblock {Weak quantitative propagation of chaos via differential calculus on
  the space of measures}.
\newblock {\em The Annals of Applied Probability}, 32(3):1929 -- 1969, 2022.

\bibitem{CHAUDRUDERAYNAL20211}
P.-E. Chaudru~de Raynal and N.~Frikha.
\newblock From the backward {K}olmogorov {PDE} on the {W}asserstein space to
  propagation of chaos for {M}c{K}ean-{V}lasov {SDE}s.
\newblock {\em Journal de Math{\'e}matiques Pures et Appliqu{\'e}es},
  156:1--124, 2021.

\bibitem{CHAUDRUDERAYNAL20221}
P.-E. Chaudru~de Raynal and N.~Frikha.
\newblock Well-posedness for some non-linear {SDE}s and related {PDE} on the
  {W}asserstein space.
\newblock {\em Journal de Math{\'e}matiques Pures et Appliqu{\'e}es},
  159:1--167, 2022.

\bibitem{crisan:murray}
D.~Crisan and E.~McMurray.
\newblock {Smoothing properties of McKean--Vlasov SDEs}.
\newblock {\em Probability Theory and Related Fields}, 171:97--148, 2018.

\bibitem{delarue:lacker:ramanan}
F.~Delarue, D.~Lacker, and K.~Ramanan.
\newblock From the master equation to mean field game limit theory: a central
  limit theorem.
\newblock {\em Electronic Journal of Probability}, pages 1--54, 2019.

\bibitem{delarue:tse}
F.~Delarue and A.~Tse.
\newblock Uniform in time weak propagation of chaos on the torus.
\newblock arXiv:2104.14973, 2021.

\bibitem{dereich:13}
S.~Dereich, M.~Scheutzow, and R.~Schottstedt.
\newblock {Constructive quantization: Approximation by empirical measures}.
\newblock {\em Annales de l'Institut Henri Poincaré, Probabilités et
  Statistiques}, 49(4):1183 -- 1203, 2013.

\bibitem{fournier2013rate}
N.~Fournier and A.~Guillin.
\newblock {On the regularisation of the noise for the Euler-Maruyama scheme
  with irregular drift}.
\newblock {\em Probability Theory and Related Fields}, 25:707--738, 2015.

\bibitem{FRIKHA202176}
N.~Frikha and L.~Li.
\newblock Well-posedness and approximation of some one-dimensional lévy-driven
  non-linear sdes.
\newblock {\em Stochastic Processes and their Applications}, 132:76--107, 2021.

\bibitem{gartner}
J.~G{\"a}rtner.
\newblock On the {M}c{K}ean-{V}lasov {L}imit for {I}nteracting {D}iffusions.
\newblock {\em Mathematische Nachrichten}, 137(1):197--248, 1988.

\bibitem{HSSzpruch:18}
W.~R.~P. Hammersley, D.~Siska, and L.~Szpruch.
\newblock {McKean–Vlasov SDEs under measure dependent Lyapunov conditions}.
\newblock {\em Annales de l'Institut Henri Poincaré, Probabilités et
  Statistiques}, 57(2):1032 -- 1057, 2021.

\bibitem{jourdain:1997}
B.~Jourdain.
\newblock Diffusions with a nonlinear irregular drift coefficient and
  probabilistic interpretation of generalized burgers' equations.
\newblock {\em ESAIM: PS}, 1:339--355, 1997.

\bibitem{kolokoltsov2010nonlinear}
Vassili~N Kolokoltsov.
\newblock {\em Nonlinear Markov processes and kinetic equations}, volume 182.
\newblock Cambridge University Press, 2010.

\bibitem{la:18}
D.~Lacker.
\newblock On a strong form of propagation of chaos for mckean-vlasov equations.
\newblock {\em Electron. Commun. Probab.}, 23:11 pp., 2018.

\bibitem{la:23}
D.~Lacker.
\newblock Hierarchies, entropy, and quantitative propagation of chaos for mean
  field diffusions.
\newblock {\em Probability and mathematical physics}, 4:377--432, 2023.

\bibitem{Li:min:2}
J.~Li and H.~Min.
\newblock Weak solutions of mean-field stochastic differential equations and
  application to zero-sum stochastic differential games.
\newblock {\em SIAM Journal on Control and Optimization}, 54(3):1826--1858,
  2016.

\bibitem{lecture:lions:college}
P.-L. Lions.
\newblock {Cours au coll{\`e}ge de France}.
\newblock
  \url{http://www.college-de-france.fr/site/pierre-louis-lions/seminar-2014-11-14-11h15.htm},
  2014.

\bibitem{mckean1967propagation}
H.~P. McKean.
\newblock Propagation of chaos for a class of non-linear parabolic equations.
\newblock {\em Stochastic Differential Equations (Lecture Series in
  Differential Equations, Session 7, Catholic Univ., 1967)}, pages 41--57,
  1967.

\bibitem{mischler2013kac}
St{\'e}phane Mischler and Cl{\'e}ment Mouhot.
\newblock Kac’s program in kinetic theory.
\newblock {\em Inventiones mathematicae}, 193:1--147, 2013.

\bibitem{mischler2015new}
St{\'e}phane Mischler, Cl{\'e}ment Mouhot, and Bernt Wennberg.
\newblock A new approach to quantitative propagation of chaos for drift,
  diffusion and jump processes.
\newblock {\em Probability Theory and Related Fields}, 161:1--59, 2015.

\bibitem{mishura:veretenikov}
Y.~Mishura and A.~Veretennikov.
\newblock Existence and uniqueness theorems for solutions of mckean--vlasov
  stochastic equations.
\newblock {\em Theory of Probability and Mathematical Statistics}, 103:59--101,
  2020.

\bibitem{Sznitman}
A.-S. Sznitman.
\newblock Topics in propagation of chaos.
\newblock In Paul-Louis Hennequin, editor, {\em Ecole d'Et{\'e} de
  Probabilit{\'e}s de Saint-Flour XIX --- 1989}, pages 165--251, Berlin,
  Heidelberg, 1991. Springer Berlin Heidelberg.

\bibitem{szpruch:tse}
L.~Szpruch and A.~Tse.
\newblock {Antithetic multilevel sampling method for nonlinear functionals of
  measure}.
\newblock {\em The Annals of Applied Probability}, 31(3):1100 -- 1139, 2021.

\end{thebibliography}

\end{document}